\documentclass[12 pt]{article}
\usepackage{amsmath,amsfonts,amsthm,amssymb,mathrsfs,latexsym,paralist,xy}
\usepackage{tikz}
\usepackage{soul}

\newtheorem{thm}{Theorem}[section]
\newtheorem{prop}[thm]{Proposition}

\newtheorem{lemma}[thm]{Lemma}
\theoremstyle{remark}
\newtheorem{rmk}[thm]{Remark}
\newtheorem{example}[thm]{Example}
\theoremstyle{definition}
\newtheorem{defn}[thm]{Definition}

\newcommand{\bi}{\begin{itemize}}
\newcommand{\ei}{\end{itemize}}
\newcommand{\be}{\begin{enumerate}}
\newcommand{\ee}{\end{enumerate}}

\newcommand{\C}{\mathbb{C}}

\renewcommand{\H}{\mathcal{H}}

\newcommand{\R}{\mathbb{R}}
\newcommand{\N}{\mathbb{N}}
\newcommand{\Z}{\mathbb{Z}}

\newcommand{\BH}{\mathcal{B}(\mathcal{H})}

\providecommand{\keywords}[1]{{\textit{Keywords and phrases:}} #1}
\providecommand{\classification}[1]{{\textit{2010 Mathematics Subject Classification:}} #1}

\def\IoIIdimdots(#1/#2/#3,#4){\node at (#1,#4) {$.$};\node at (#2,#4) {$.$};\node at (#3,#4) {$.$};}
\def\IIoIIdimdots(#1,#2/#3/#4){\node at (#1,#2) {$.$};\node at (#1,#3) {$.$};\node at (#1,#4) {$.$};}

\def\IoIIIdimdots(#1/#2/#3,#4,#5){\node at (#1,#4,#5) {$.$};\node at (#2,#4,#5) {$.$};\node at (#3,#4,#5) {$.$};}
\def\IIoIIIdimdots(#1,#2/#3/#4,#5){\node at (#1,#2,#5) {$.$};\node at (#1,#3,#5) {$.$};\node at (#1,#4,#5) {$.$};}
\def\IIIoIIIdimdots(#1,#2,#3/#4/#5){\node at (#1,#2,#3) {$.$};\node at (#1,#2,#4) {$.$};\node at (#1,#2,#5) {$.$};}

\usepackage[margin=1in]{geometry}
\AtEndDocument{\bigskip{\small%
  \textsc{Carla Farsi, Judith Packer : Department of Mathematics, University of Colorado at Boulder, Boulder, CO 80309-0395, USA.} \par
  \textit{E-mail address}: \texttt{carla.farsi@colorado.edu, packer@euclid.colorado.edu} \par
  \addvspace{\medskipamount}
  \textsc{Elizabeth Gillaspy : Department of Mathematics, University of Montana, 32 Campus Drive \#0864, Missoula, MT 59812-0864, USA.}\par
  \textit{E-mail address}: \texttt{elizabeth.gillaspy@mso.umt.edu}\par
  \addvspace{\medskipamount}
  \textsc{Palle Jorgensen : Department of Mathematics, 14 MLH, University of Iowa, Iowa City, IA 52242-1419, USA.}\par
  \textit{E-mail address}: \texttt{palle-jorgensen@uiowa.edu}\par
   \addvspace{\medskipamount}
  \textsc{Sooran Kang : College of General Education, Chung-Ang University, 84 Heukseok-ro, Dongjak-gu, Seoul, Republic of Korea.} \par
  \textit{E-mail address}, \texttt{sooran09@cau.ac.kr}
}}

\begin{document}

\title{Monic representations of finite higher-rank graphs}
\author{Carla Farsi, Elizabeth Gillaspy, Palle Jorgensen,  Sooran Kang, and Judith Packer}
\date{\today}
\maketitle

\begin{abstract} 
In this paper   we  define the notion of monic representation for the $C^*$-algebras of finite higher-rank graphs with no sources, and undertake a comprehensive study of them.   Monic representations are the representations  that, when restricted to the commutative $C^*$-algebra of the continuous functions on  the infinite path space, admit a cyclic vector. We {link monic representations to the $\Lambda$-semibranching representations previously studied by Farsi, Gillaspy, Kang, and Packer, and }also provide  a  universal  representation model for nonnegative monic representations.

\end{abstract}

\classification{46L05, 46L55, 46K10}


\tableofcontents

\section{Introduction}

Higher-rank graphs $\Lambda$ -- also known as $k$-graphs -- and their $C^*$-algebras $C^*(\Lambda)$ were introduced by Kumjian and Pask in \cite{KP}, building on the work of Robertson and Steger \cite{RS-London, RS}. 
Generalizations of the 
Cuntz--Krieger $C^*$-algebras associated to directed graphs (cf.~\cite{Cuntz, Cuntz-Krieger, enomoto-watatani,kpr}), $k$-graph  $C^*$-algebras  
share many of the important properties of Cuntz and Cuntz--Krieger  $C^*$-algebras, including Cuntz--Krieger uniqueness theorems and realizations as 
groupoid $C^*$-algebras.
Moreover, the $C^*$-algebras of higher-rank graphs  are closely linked with orbit equivalence for shift spaces \cite{carlsen-ruiz-sims} and with symbolic dynamics more generally \cite{pask-raeburn-weaver,skalski-zacharias-houston, pask-sierakowski-sims}, as well as with fractals and self-similar structures \cite{FGJKP1, FGJKP2}. More links between higher-rank graphs and symbolic dynamics can be seen via \cite{bezuglyi-k-m, bezuglyi-four-auth} and the references cited therein. 
 
 The research presented in the pages that follow develops a non-commutative harmonic analysis for finite higher-rank graphs with no sources.  More precisely, we  introduce monic 
 representations for the $C^*$-algebras associated to finite higher-rank graphs with no sources; undertake a detailed theoretical analysis of such representations;
 and  present a variety of examples.

Like the Cuntz--Krieger algebras,  $k$-graph $C^*$-algebras often fall in a class of non-type I, and in fact purely infinite $C^*$-algebras. The significance of this for representation theory is that the unitary equivalence classes of irreducible representations of $k$-graph $C^*$-algebras  do not arise as Borel cross sections  \cite{Glimm1,Glimm2,Dixmier,Effros1,Effros2}. 
  In short, for these $C^*$-algebras, only subfamilies of irreducible representations admit ``reasonable'' parametrizations. 
  
  Various specific subclasses of representations of Cuntz and Cuntz--Krieger $C^*$-algebras have been extensively studied by many researchers, who were motivated by their applicability to a wide variety of fields.  In addition to  connections with wavelets  (cf. \cite{dutkay-jorgensen-mart,dutkay-jorgensen-four,MP,FGKP, FGKPexcursions,FGJKP2}), representations of Cuntz--Krieger algebras have been linked to  fractals and  Cantor sets \cite{Str,kawamura-2016,FGJKP1,FGJKP2} and to the endomorphism group  of a Hilbert space \cite{bratteli-jorgensen-price,Laca-Thesis}. Indeed, the astonishing goal of identifying both discrete and continuous series of representations of Cuntz (and to some extent Cuntz--Krieger) $C^*$-algebras,  was accomplished in \cite{dutkay-jorgensen-monic,dutkay-jorgensen-atomic,bezuglyi-jorgensen}, building on the pioneering results of \cite{BJ-atomic}.

In the setting of higher-rank graphs, however, the representation theory of these $C^*$-algebras is in its infancy.  
Although the primitive ideal space of higher-rank graph $C^*$-algebras is well understood \cite{CKSS, Kang-Pask},    representations of $k$-graph $C^*$-algebras have only been systematically studied in the one-vertex case \cite{dav-pow-yan-atomic,Yang-End,davidson-yang-represent}. 
This motivated us to undertake 
the present detailed study of monic representations of $k$-graph $C^*$-algebras and their unitary equivalence classes.   Despite the similarities between the Cuntz algebras and $k$-graph $C^*$-algebras which we have highlighted above, there are   fundamental structural differences between them: for example, $k$-graph $C^*$-algebras need not be simple, nor is their $K$-theory known in general.  Thus, the extension of  results on representations for Cuntz algebras to the $k$-graph context is not automatic, and we are pleasantly surprised to have obtained such extensions in the pages that follow.

  The monic representations that we focus on in this paper were inspired in part by the wavelet theory for higher-rank graphs which was developed in \cite{FGKP}.  These wavelets relied on the concept of $\Lambda$-semibranching function systems, introduced in \cite{FGKP} and further studied in \cite{FGJorKP}.  In this paper,  we refine the $\Lambda$-semibranching function systems into a crucial technical tool for studying monic representations, namely the $\Lambda$-projective systems of Definition \ref{def:lambda-proj-system}.  
  Monic representations also have strong connections to  Markov measures \cite{dutkay-jorgensen-monic, bezuglyi-jorgensen}  and Nelson's universal representation of an abelian algebra \cite{nelson}.  Indeed,  studying monic representations enables us to convert questions about the representation theory of higher-rank graphs into measure-theoretic questions (see Theorems \ref{thm-disjoint-monic-repres} and \ref{thm-irred-monic-repres} below).

  This paper is organized as follows.  We begin with an introductory section which reviews the basic notation and terminology for higher-rank graphs, as well as the $\Lambda$-semibranching function systems from \cite{FGKP}.
  Before turning our attention to a theoretic and systematic analysis of the monic  representations of finite $k$-graph $C^*$-algebras, Section \ref{sec-repres-k-graphs} develops
   the technical tools we will need for this analysis.  The $\Lambda$-semibranching function systems of \cite{FGKP} are refined in Section \ref{sec:lambda-proj-systems} into  $\Lambda$-projective systems, and  Section \ref{sec:Proj-valued-measures} analyzes the projection-valued measure $P = P_\pi$  on the infinite path space $\Lambda^\infty$ which arises 
from  a representation $\pi$ of $C^*(\Lambda)$.

   Section  \ref{sec-repres-k-graphs} ends by addressing the question of when representations of $k$-graph $C^*$-algebras are disjoint or irreducible, see Theorems  \ref{thm-disjoint-monic-repres} and
       \ref{thm-irred-monic-repres}. To be precise, 
     Theorem \ref{thm-disjoint-monic-repres} 
       shows that for representations of $C^*(\Lambda)$ arising from $\Lambda$-projective systems, the   task of checking when two representations are 
        equivalent reduces to a measure-theoretical problem.
 Theorem \ref{thm-irred-monic-repres} characterizes the commutant of such representations, enabling a precise description of when a representation arising from a $\Lambda$-projective system is irreducible.

  Having laid the necessary technical groundwork, we undertake the promised analysis of  monic representations of $C^*(\Lambda)$  in Section \ref{sec:monic-results}.  This section contains two  of the main results of this paper as well as a number of examples of monic representations.
  Theorem \ref{thm-characterization-monic-repres} establishes that, when $\Lambda$ is a  finite and  source-free $k$-graph,  monic representations of $C^*(\Lambda)$ are always unitarily equivalent to a $\Lambda$-projective representation on $\Lambda^\infty$. 
   Theorem \ref{prop-conv-to-8-7}  
   gives an alternative, measure-theoretic characterization of when a $\Lambda$-semibranching function system gives rise to a monic representation.  More precisely, Theorem \ref{prop-conv-to-8-7} proves that a $\Lambda$-semibranching representation is monic if and only if the measure-theoretic subsets specified by the $\Lambda$-semibranching function system (see Definition \ref{def-lambda-SBFS-1} below)
  generate the $\sigma$-algebra.

The final section,  Section \ref{sec:univ_repn}, relates monic representations to Nelson's universal Hilbert space, which we denote $\H(\Lambda^\infty)$.  Theorem \ref{prop-universal-name} shows that
  every monic representation whose associated $\Lambda$-projective system consists of positive functions is unitarily equivalent to a sub-representation of the so-called ``universal representation'' of $C^*(\Lambda)$ on $\H(\Lambda^\infty)$ which is described in Proposition \ref{prop-univ-repres}.  
In particular, this Theorem establishes that every representation of $C^*(\Lambda)$ which arises from a $\Lambda$-semibranching function   system, as in \cite{FGKP}, is unitarily equivalent to a sub-representation of the universal representation.

\subsection*{Acknowledgments}   	 
The authors would like to thank Daniel Gon\c calves, Janos Englander and Alex Kumjian  for helpful discussions.  
E.G.~was partially supported by   the Deutsches Forschungsgemeinschaft via the SFB 878 ``Groups, Geometry, and Actions'' of the Universit\"at M\"unster. 	 
S.K.~was supported by Basic Science Research Program through the National Research Foundation of Korea (NRF) funded by the Ministry of Education (\#2017R1D1A1B03034697).
C.F. and J.P.~were  partially supported by two individual  grants from the Simons Foundation (C.F. \#523991; J.P. \#316981).
P.J. thanks his colleagues in the Math Department at the University of Colorado, for making a week-long visit there possible, for support, and for kind hospitality.  
Progress towards the completion of this manuscript was made by the first three named co-authors while in attendance at the Fields Institute (Toronto) and the Mathematical Congress of the Americas (Montreal) in 2017; we are grateful for their support of our collaboration.
C.F. also thanks IMPAN for hospitality during her visits to IMPAN, Warsaw, Poland, where some of this work was carried out (grant \#3542/H2020/2016/2). This paper was partially supported by the grant H2020-MSCA-RISE-2015-691246-QUANTUM DYNAMICS.

\section{Foundational material}
\label{sec-fund-mater}

\subsection{Higher-rank graphs}

 We  recall the definition of  higher-rank graphs and their $C^*$-algebras from  \cite{KP}. 
  
Let $\N=\{0,1,2,\dots\}$ denote the monoid of natural numbers under addition, and let $k\in \N$ with $k\ge 1$. We write $e_1,\dots e_k$ for the standard basis vectors of $\N^k$, where $e_i$ is the vector of $\N^k$ with $1$ in the $i$-th position and $0$ everywhere else.

\begin{defn} \cite[Definition 1.1]{KP}
\label{def-higher-rank-gr}
A \emph{higher-rank graph} or \emph{$k$-graph} is a countable small category\footnote{Recall that a small category is one  in which  the collection of arrows is a  set.} $\Lambda$ 
 with a degree functor $d:\Lambda\to \N^k$ satisfying the \emph{factorization property}: for any morphism $\lambda\in\Lambda$ and any $m, n \in \N^k$ such that  $d(\lambda)=m+n \in \N^k$,  there exist unique morphisms $\mu,\nu\in\Lambda$ such that $\lambda=\mu\nu$ and $d(\mu)=m$, $d(\nu)=n$. 
\end{defn}



When discussing $k$-graphs, we use the arrows-only picture of category theory; thus, objects in $\Lambda$ are identified with identity morphisms, and the notation $ \lambda \in \Lambda $ means $\lambda$ is a morphism in $\Lambda$.
We often regard $k$-graphs as a generalization of directed graphs, so we call morphisms $\lambda\in\Lambda$ \emph{paths} in $\Lambda$, and the objects (identity morphisms) are often called \emph{vertices}. For $n\in\N^k$
, we write
\begin{equation}
\label{eq:Lambda-n}
\Lambda^n:=\{\lambda\in\Lambda\,:\, d(\lambda)=n\}\
\end{equation}
With this notation, note that $\Lambda^0$ is the set of objects (vertices) of $\Lambda$, and we will call elements of $\Lambda^{e_i}$ (for any $i$) \emph{edges}.  
We write $r,s:\Lambda\to \Lambda^0$ for the range and source maps in $\Lambda$ respectively.  For vertices $v, w \in \Lambda^0$, we define
\[v\Lambda w:=\{\lambda\in\Lambda\,:\, r(\lambda)=v,\;s(\lambda)=w\} \quad \text{and} \quad  v\Lambda^n:= \{ \lambda \in \Lambda: r(\lambda) = v, \ d(\lambda) = n\}.\]

Our focus in this paper is on  finite $k$-graphs with no sources.
A $k$-graph $\Lambda$ is \emph{finite} if $\Lambda^n$ is a finite set for all $n\in\N^k$.
We say that $\Lambda$  \emph{has no sources} or \emph{is source-free} if $v\Lambda^n\ne \emptyset$ for all $v\in\Lambda^0$ and $n\in\N^k$. It is well known that this is equivalent to the condition that $v\Lambda^{e_i}\ne \emptyset$ for all $v\in \Lambda$ and all basis vectors $e_i$ of $\N^k$. 

For $m,n\in\N^k$, we write $m\vee n$ for the coordinatewise maximum of $m$ and $n$. Given  $\lambda,\eta\in \Lambda$, we write
\begin{equation}\label{eq:lambda_min}
\Lambda^{\operatorname{min}}(\lambda,\eta):=\{(\alpha,\beta)\in\Lambda\times\Lambda\,:\, \lambda\alpha=\eta\beta,\; d(\lambda\alpha)=d(\lambda)\vee d(\eta)\}.
\end{equation}
If $k=1$, then $\Lambda^{\operatorname{min}}(\lambda, \eta)$ will have at most one element; this need not be true 
if $k > 1$.

For finite source-free $k$-graphs $\Lambda$, for each $1 \leq i \leq k$, we can define the \emph{$i$th vertex matrix} $A_i \in M_{\Lambda^0}(\N)$ by 
\begin{equation} A_i(v,w) = |v\Lambda^{e_i} w|.\label{eq:adj-mx}
\end{equation}
Observe that the factorization property implies that $A_iA_j=A_jA_i$ for $1\le i,j\le k$. 

%
{We now describe two fundamental examples of higher-rank graphs which were first mentioned in the foundational paper \cite{KP}.}  More examples of higher-rank graphs can be found in Section \ref{sec:monic-examples} below.
\begin{example}
\begin{itemize}
\item[(a)] For any directed graph $E$, let $\Lambda_E$ be the category whose objects are the vertices of $E$ and whose morphisms are the finite paths in $E$. Then $\Lambda_E$ is a 1-graph whose degree functor $d:\Lambda_E\to \N$ 
is given by $d(\eta) = |\eta|$ (the number of edges in $\eta$).
\item[(b)] For $k\ge 1$, let $\Omega_k$ be the small category with
\[
\operatorname{Obj}(\Omega_k)=\N^k,\quad \text{and}\quad \operatorname{Mor}(\Omega_k)=\{(p,q)\in \N^k\times \N^k\,:\, p\le q\}.
\]
Again, we can also view elements of $\text{Obj}(\Omega_k)$ as identity morphisms, via the map $\text{Obj}(\Omega_k) \ni p \mapsto (p, p) \in \text{Mor}(\Omega_k)$.
The range and source maps $r,s:\operatorname{Mor}(\Omega_k)\to \operatorname{Obj}(\Omega_k)$ are given by $r(p,q)=p$ and $s(p,q)=q$. If we define $d:\Omega_k\to \N^k$ by $d(p,q)=q-p$, then one can check that $\Omega_k$ is a $k$-graph with degree functor $d$.
\end{itemize}

\end{example}

\begin{defn}[\cite{KP} Definitions 2.1]
\label{def:infinite-path}
Let $\Lambda$ be a $k$-graph. An \emph{infinite path} in $\Lambda$ is a $k$-graph morphism (degree-preserving functor) $x:\Omega_k\to \Lambda$, and we write $\Lambda^\infty$ for the set of infinite paths in $\Lambda$. 
Since $\Omega_k$ has a terminal object (namely $0 \in\N^k$) but no initial object, we think of our infinite paths as having a range $r(x) : = x(0)$ but no source.
For each $m\in \N^k$, we have a shift map $\sigma^m:\Lambda^\infty \to \Lambda^\infty$ given by
\begin{equation}\label{eq:shift-map}
\sigma^m(x)(p,q)=x(p+m,q+m)
\end{equation}
for $x\in\Lambda^\infty$ and $(p,q)\in\Omega_k$. 

It is well-known that the collection of cylinder sets 
\[
Z(\lambda)=\{x\in\Lambda^\infty\,:\, x(0,d(\lambda))=\lambda\},
\]
for $\lambda \in \Lambda$, form a compact open basis for a locally compact Hausdorff topology on $\Lambda^\infty$, under  reasonable hypotheses on $\Lambda$ (in particular, when $\Lambda$ is row-finite: see Section 2 of \cite{KP}). If $\Lambda$ is  finite, then $\Lambda^\infty$ is compact in this topology.

{We also have a partially defined ``prefixing map'' $\sigma_\lambda: Z(r(\lambda)) \to Z(\lambda)$ for each $\lambda \in \Lambda$:
\[ \sigma_\lambda(x) = \lambda x = \left[ (p, q) \mapsto \begin{cases} \lambda(p, q), & q \leq d(\lambda) \\
x(p-d(\lambda), q-d(\lambda)), & p \geq d(\lambda) \\
\lambda (p, d(\lambda)) x(0, q-d(\lambda)), & p < d(\lambda) < q
\end{cases} \right]
\]}
\end{defn}
{
\begin{rmk}
The factorization rule implies an important property of infinite paths:  for any $x\in\Lambda^\infty$ and $m\in\N^k$, we have
\[
x=x(0,m)\sigma^m(x).
\]
Taking $m = p e_j$ for an arbitrary $p \in \N$ reveals that every infinite path must contain infinitely many edges of each color. 
Moreover, if we take $m = (n, n, \ldots, n) \in \N^k$ for some $n \geq 1$, the factorization rule tells us that $x(0,m)$ can be written uniquely as a ``rainbow sequence'' of edges:
\[ x(0,m) = f_1^1 f_2^1 \cdots f_k^1 f_1^2 \cdots f_k^2 f_1^3 \cdots  f_k^n,\]
where $d(f_i^j) = e_i$.

For example, suppose $\Lambda$ is a 2-graph.  We can visualize $\Lambda$ as arising from a 2-colored graph (red and blue edges).  Moreover, each infinite path $x\in\Lambda^\infty$ can be uniquely identified with  an infinite string of alternating blue and red edges (setting blue to be ``color 1'' and red to be ``color 2'').  
\label{rmk:rainbow}
\end{rmk}
We stress that even finite $k$-graphs may have nontrivial infinite paths; in an infinite path, the same edge may occur multiple times and even infinitely many times.}

Now we introduce the $C^*$-algebra associated to a finite, source-free $k$-graph $\Lambda$. 
\begin{defn}\label{def:kgraph-algebra}
Let $\Lambda$ be a finite $k$-graph with no sources. A \emph{Cuntz--Krieger $\Lambda$-family} is a collection  $\{t_\lambda:\lambda\in\Lambda\}$ of partial isometries in a $C^*$-algebra satisfying
\begin{itemize}
\item[(CK1)] $\{t_v\,:\, v\in\Lambda^0\}$ is a family of mutually orthogonal projections,
\item[(CK2)] $t_\lambda t_\eta=t_{\lambda\eta}$ if $s(\lambda)=r(\eta)$,
\item[(CK3)] $t^*_\lambda t_\lambda=t_{s(\lambda)}$ for all $\lambda\in\Lambda$,
\item[(CK4)] for all $v\in\Lambda$ and $n\in\N^k$, we have $
t_v=\sum_{\lambda\in v\Lambda^n} t_\lambda t^*_\lambda.$
\end{itemize}
The Cuntz--Krieger $C^*$-algebra $C^*(\Lambda)$ associated to $\Lambda$ is the universal $C^*$-algebra generated by a Cuntz--Krieger $\Lambda$-family.
\end{defn}
The condition (CK4) implies that for all $\lambda, \eta\in \Lambda$, we have
\begin{equation}\label{eq:CK4-2}
t_\lambda^* t_\eta=\sum_{(\alpha,\beta)\in \Lambda^{\operatorname{min}}(\lambda,\eta)} t_\alpha t^*_\beta. 
\end{equation}
It follows that $
C^*(\Lambda)=\overline{\operatorname{span}}\{t_\alpha t^*_\beta\,:\, \alpha,\beta\in\Lambda,\; s(\alpha)=s(\beta)\}.$


\subsection{$\Lambda$-semibranching function systems and their representations}

In \cite{FGKP}, separable representations of $C^*(\Lambda)$ were constructed by using $\Lambda$-semibranching function systems on measure spaces. A $\Lambda$-semibranching function system is a generalization of the semibranching function systems studied by Marcolli and Paolucci in \cite{MP}.  
As established in \cite{MP, FGKP},  $\Lambda$-semibranching function systems (and their one-dimensional counterparts) give rise to representations of $C^*(\Lambda)$, {and we provide examples of such representations in Section \ref{sec:monic-examples} below}.  Indeed, we build upon the notion of $\Lambda$-semibranching function systems in Sections \ref{sec-repres-k-graphs}  
and \ref{sec:monic-results} below to characterize the monic representations of higher-rank graphs.



\begin{defn}
\label{def-1-brach-system}\cite[Definition~2.1]{MP}\label{defn:sbfs}
Let $(X,\mu)$ be a measure space. Suppose that, for each $1\le i\le N$, we have a measurable map $\sigma_i:D_i\to X$, for some measurable subsets $D_i\subset X$. The family $\{\sigma_i\}_{i=1}^N$ is a \emph{semibranching function system} if the following holds:
\begin{itemize}
\item[(a)] Setting $R_i = \sigma_i(D_i),$ we have 
\[
\mu(X\setminus \cup_i R_i)=0,\quad\quad\mu(R_i\cap R_j)=0\;\;\text{for $i\ne j$}.
\]
\item[(b)] For each $i$, the Radon-Nikodym derivative
\[
\Phi_{\sigma_i}=\frac{d(\mu\circ\sigma_i)}{d\mu}
\]
satisfies $\Phi_{\sigma_i}>0$, $\mu$-almost everywhere on $D_i$.
\end{itemize}
A measurable map $\sigma:X\to X$ is called a \emph{coding map} for the family $\{\sigma_i\}_{i=1}^N$ if $\sigma\circ\sigma_i(x)=x$ for all $x\in D_i$.
\end{defn}

\begin{defn}\cite[Definition~3.2]{FGKP}
\label{def-lambda-SBFS-1}
Let $\Lambda$ be a finite $k$-graph and let $(X, \mu)$ be a measure space.  A \emph{$\Lambda$-semibranching function system} on $(X, \mu)$ is a collection $\{D_\lambda\}_{\lambda \in \Lambda}$ of measurable subsets of $X$, together with a family of prefixing maps $\{\tau_\lambda: D_\lambda \to X\}_{\lambda \in \Lambda}$, and a family of coding maps $\{\tau^m: X \to X\}_{m \in \N^k}$, such that
\begin{itemize}
\item[(a)] For each $m \in \N^k$, the family $\{\tau_\lambda: d(\lambda) = m\}$ is a semibranching function system, with coding map $\tau^m$.
\item[(b)] If $ v \in \Lambda^0$, then  $\tau_v = id$,  and $\mu(D_v) > 0$.
\item[(c)] Let $R_\lambda = \tau_\lambda( D_\lambda)$. For each $\lambda \in \Lambda, \nu \in s(\lambda)\Lambda$, we have $R_\nu \subseteq D_\lambda$ (up to a set of measure 0), and
\[\tau_{\lambda} \tau_\nu = \tau_{\lambda \nu}\text{ a.e.}\]
 (Note that this implies that up to a set of measure 0, $D_{\lambda \nu} = D_\nu$ whenever $s(\lambda) = r(\nu)$).
\item[(d)] The coding maps satisfy $\tau^m \circ \tau^n = \tau^{m+n}$ for any $m, n \in \N^k$.  (Note that this implies that the coding maps pairwise commute.)
\end{itemize}
\end{defn}

\begin{rmk}
\label{rmk:abs-cts-inverse}
We pause to note that  condition (c) of Definition \ref{def-lambda-SBFS-1} above implies that $D_\lambda=D_{s(\lambda)}$ and $R_\lambda\subset R_{r(\lambda)}$ for $\lambda\in\Lambda$. Also, when $\Lambda$ is a finite 1-graph, the definition of a $\Lambda$-semibranching function system is not equivalent to Definition \ref{def-1-brach-system}. In particular, Definition~\ref{def-lambda-SBFS-1}(b) implies that the domain sets $\{D_v:v\in\Lambda^0\}$ must satisfy $\mu(D_v\cap D_w)=\mu(R_v\cap R_w)=0$ for $v\ne w\in\Lambda^0$, but  Definition~\ref{def-1-brach-system} does not require that the domain sets $D_i$ be mutually disjoint $\mu$-a.e. In fact, Definition~\ref{def-lambda-SBFS-1} implies what is called condition (C-K) in Section~2.4 of \cite{bezuglyi-jorgensen}: up to a measure zero set, 
\begin{equation}
\label{eq-partition}
D_v=\bigcup_{\lambda\in v\Lambda^m} R_\lambda
\end{equation}
for all $v\in\Lambda^0$ and $m\in\N,$ since $R_v=\tau_v(D_v)=id(D_v)=D_v.$ Also notice that in the above decomposition the  intersections $R_\lambda \cap R_{\lambda'},$ $R_\lambda \not= {\lambda'},$ have measure zero. This condition  is crucial to making sense of the representation of $C^*(\Lambda)$ associated to the $\Lambda$-semibranching function system (see Theorem \ref{thm:separable-repn} below). As established in Theorem~2.22 of \cite{bezuglyi-jorgensen}, in order to obtain a representation of a 1-graph algebra $C^*(\Lambda)$ from a semibranching function system, one must also assume that the semibranching function system satisfies condition (C-K).
\end{rmk}

We pause to enumerate some properties of $\Lambda$-semibranching function systems, which can be proved by routine computations. 

\begin{rmk}
 \begin{enumerate}
 \item For any $n\in \N^k$ and any measurable $E\subseteq X$, we have  
\begin{equation} 
(\tau^n)^{-1} (E) = \bigcup_{\lambda \in \Lambda^n} \tau_\lambda(E) \quad \text{ and consequently } \quad 
\mu \circ (\tau^n)^{-1} << \mu
\end{equation}
in any $\Lambda$-semibranching function system.

 \item On $R_\lambda$, we have $(\tau_\lambda)^{-1} = \tau^n$.  Therefore, Condition (b) of Definition \ref{def-1-brach-system} implies that 
 $\mu \circ (\tau_\lambda)^{-1} << \mu$ on $R_\lambda$, and $ \frac{d(\mu \circ (\tau_\lambda)^{-1})}{d\mu}$ is  nonzero a.e.~on $R_\lambda$. 


 \end{enumerate}
 \label{rmk:lambda-sbfs-properties}
 \end{rmk}

As established in \cite{FGKP}, any $\Lambda$-semibranching function system gives rise to a representation of $C^*(\Lambda)$ { via \lq prefixing' and \lq chopping off' operators that satisfy the Cuntz-Krieger relations}.  {Intuitively, a $\Lambda$-semibranching function system is a way of encoding the Cuntz-Krieger relations at the measure-space level: the prefixing map $\tau_\lambda$ corresponds to the partial isometry $s_\lambda \in C^*(\Lambda)$.  
For the convenience of the reader, we recall the formula for these $\Lambda$-semibranching representations of $C^*(\Lambda)$.

\begin{thm}\cite[Theorem~3.5]{FGKP}\label{thm:separable-repn}
Let $\Lambda$ be a finite $k$-graph with no sources and suppose that we have a $\Lambda$-semibranching function system on a  measure space $(X,\mu)$ with prefixing maps $\{\tau_\lambda: \lambda \in \Lambda\}$ and coding maps $\{\tau^m:m\in \N^k\}$. For each $\lambda\in\Lambda$, define an operator $S_\lambda$  on $L^2(X,\mu)$ by
\[
S_\lambda\xi(x)=\chi_{R_\lambda}(x)(\Phi_{\tau_\lambda}(\tau^{d(\lambda)}(x)))^{-1/2} \xi(\tau^{d(\lambda)}(x)).
\]
Then the operators $\{S_\lambda:\lambda\in\Lambda\}$ generate a representation $\pi$ of $C^*(\Lambda)$ on $L^2(X, \mu)$.
\label{thm:SBFS-repn}
\end{thm}

\section{Representations of higher-rank graph $C^*$-algebras:  first analysis} 
\label{sec-repres-k-graphs}

We begin this section by 
developing the technical tools which we will rely on throughout the paper: $\Lambda$-projective systems and projection valued measures.
These tools enable us 
 to describe 
when certain representations of $k$-graph $C^*$-algebras are disjoint or irreducible, see Theorems  \ref{thm-disjoint-monic-repres} and
   \ref{thm-irred-monic-repres}. 


\subsection{$\Lambda$-projective systems and representations}
\label{sec:lambda-proj-systems}



The definition of a $\Lambda$-projective system generalizes to the $k$-graph setting the definition of a  monic system in \cite{dutkay-jorgensen-monic} (for the Cuntz algebras $\mathcal{O}_N$) and \cite{bezuglyi-jorgensen} (in the case of Cuntz--Krieger algebras $\mathcal{O}_A$).  We have decided to change the name because even for $\mathcal{O}_A$, not every monic system gives rise to a monic representation of $\mathcal{O}_A$. 
The word ``projective'' refers to the cocycle-like Condition (b) of Definition \ref{def:lambda-proj-system}.

\begin{defn}\label{def:lambda-proj-system}
Let $\Lambda$ be a finite $k$-graph with no sources. 
A \emph{$\Lambda$-projective system} on a measure space $(X,\mu)$ is a $\Lambda$-semibranching function system on $(X,\mu)$, with prefixing maps $\{\tau_\lambda:D_\lambda\to R_\lambda\}_{\lambda\in\Lambda}$ and coding maps $\{\tau^n:n\in\N^k\}$
together with a family of functions $\{ f_\lambda\}_{\lambda\in \Lambda} \subseteq  L^2(X,\mu)$ satisfying the following conditions:
\begin{itemize}
\item[(a)] For any   $\lambda \in \Lambda$,   we have  $ 0\not=\frac{d(\mu \circ (\tau_\lambda)^{-1})}{d\mu} = |f_\lambda |^2$; 
\item[(b)] For any $\lambda, \nu \in \Lambda$, we have 
$f_\lambda \cdot (f_\nu \circ \tau^{d(\lambda)}) = f_{\lambda \nu}.$
\end{itemize}
\end{defn}

 Thus, a $\Lambda$-projective system on $(X,\mu)$ consists of a $\Lambda$-semibranching function system plus some extra information (encoded in the functions $f_\lambda$).  We have a certain amount of choice for the functions $f_\lambda$; we can take positive or negative (or imaginary!) roots of $\frac{d(\mu \circ (\tau_\lambda)^{-1})}{d\mu} $ for $f_\lambda$, as long as they satisfy the multiplicativity Condition (b) above.

 \begin{example} 
 For any $\Lambda$-semibranching function system on $(X, \mu)$, there is a natural choice of an associated $\Lambda$-projective system; namely, for $\lambda \in \Lambda^{n}$ we define 
 \begin{equation}
 \label{eq:std-f-lambda}
 f_\lambda(x) := \Phi_{\lambda}(\tau^{n}(x))^{-1/2}  \chi_{R_\lambda}(x).\end{equation}
Condition (a) is satisfied because of the hypothesis that the Radon--Nikodym derivatives be strictly positive $\mu$-a.e.~on their domain of definition.
\label{ex:standard-SBFS-is-proj}
Since the operators $S_\lambda \in B(L^2(X, \mu))$ of Theorem \ref{thm:SBFS-repn} are given by 
\[S_\lambda(f) = f_\lambda \cdot (f \circ \tau^n),\]
and \cite[Theorem 3.5]{FGKP} establishes that $\{S_\lambda\}_{\lambda\in \Lambda}$ is a Cuntz--Krieger family, Proposition \ref{prop:lambda-proj-repn} below shows that Equation \eqref{eq:std-f-lambda}  indeed describes a $\Lambda$-projective system. 
\end{example}

  \begin{rmk}
 \label{rmk:lambda-proj-comments-1}
 Observe that Condition (a) of Definition \ref{def:lambda-proj-system}
 forces $f_\lambda (x) = 0 $ a.e.~outside of $R_\lambda$, since $\frac{d(\mu \circ (\tau_\lambda)^{-1})}{d\mu}$ is supported only on $R_\lambda$. 
\end{rmk}

 Condition (b) of Definition~\ref{def:lambda-proj-system} is needed to associate a representation of $C^*(\Lambda)$ to a $\Lambda$-projective system.  To be precise, we have: 
 
  \begin{prop}
  \label{prop:lambda-proj-repn}
  Let $\Lambda$ be a finite, source-free $k$-graph. Suppose that a measure space $(X,\mu)$ admits a $\Lambda$-semibranching function system with prefixing maps $\{\tau_\lambda:\lambda\in\Lambda\}$ and coding maps $\{\tau^n:n\in\N^k\}$.  Suppose that $\{f_\lambda\}_{\lambda\in \Lambda}$ is a collection of functions satisfying Condition (a) 
  of Definition \ref{def:lambda-proj-system}. Then the maps $\{\tau_\lambda\}$, $\{\tau^n\}$ and $\{f_\lambda\}_\lambda$ form a $\Lambda$-projective system on $(X, \mu)$ if and only if the operators $T_\lambda \in B(L^2(X, \mu))$ given by 
  \begin{equation}\label{eq:T-lambda}
  T_\lambda(f) = f_\lambda \cdot (f \circ \tau^{d(\lambda)})
  \end{equation}
 form a Cuntz--Krieger $\Lambda$-family with each $T_\lambda$ nonzero (and hence give a representation of $C^*(\Lambda)$).
  \end{prop}

  	\begin{proof} If the operators $T_\lambda$ of Equation \eqref{eq:T-lambda} form a nontrivial Cuntz--Krieger $\Lambda$-family, then it is easily checked that the functions $\{f_\lambda\}_{\lambda\in \Lambda}$ satisfy the hypotheses of Definition \ref{def:lambda-proj-system}.
 {
%
%
  
  On the other hand, suppose that $(X,\mu)$ admits a $\Lambda$-projective system with prefixing maps $\{\tau_\lambda\}_{\lambda \in \Lambda}$, coding maps $\{\tau^n\}_{n\in \N^k}$, and functions $\{f_\lambda\}_{\lambda \in \Lambda}$.  We will show that the operators $\{T_\lambda\}$ of Equation \eqref{eq:T-lambda} satisfy Conditions (CK1)--(CK4). 
  
For (CK1),  observe that if $v \in \Lambda^0$, $T_v(f) = f_v \cdot (f\circ \tau^0)$ is supported on $D_v = R_v$ by Condition (a) of Definition~\ref{def:lambda-proj-system}. Moreover, since  $v= v^2$ for any $v \in \Lambda^0,$ and $\tau_v = id_{D_v} = \tau^0$,  Condition (b) of Definition~\ref{def:lambda-proj-system} implies that 
\[f_v = f_v \cdot (f_v \circ \tau^0) = f_v^2 \Rightarrow f_v = \chi_{D_v}.\]
 Consequently, $T_v(f) = \chi_{D_v} \cdot f$.  Since the sets $\{D_v = R_v\}_{v\in\Lambda^0}$ are disjoint (up to a set of measure zero), it follows that $\{T_v: v\in \Lambda^0\}$ is a set of mutually orthogonal projections; in other words, (CK1) holds. 

For (CK2), fix $\lambda,\nu\in \Lambda$ with $s(\lambda)=r(\nu)$.
 Since $f_\nu(x) = 0$ unless $x \in R_\nu$, we see that 
 \begin{align*}
  T_\lambda T_\nu(f)(x) 
  &= \begin{cases}
  0, & x\not\in \tau_\lambda (R_\nu) \\
  f_\lambda(\tau_\lambda \circ \tau_\nu(y)) \cdot f_\nu(\tau_\nu(y)) \cdot f(y), & x = \tau_\lambda \tau_\nu(y).
  \end{cases}
 \end{align*}
 On the other hand,  Condition (b) of Definition \ref{def:lambda-proj-system} implies that if $x = \tau_\lambda \tau_\nu(y)$,
 \[f_\lambda(\tau_\lambda \circ \tau_\nu(y)) \cdot f_\nu(\tau_\nu(y)) = f_\lambda(x) \cdot f_\nu(\tau^{d(\lambda)}(x)) = f_{\lambda \nu}(y).\]
This implies that $T_\lambda T_\nu = T_{\lambda\nu}$ as claimed.
 
  To check (CK3),
 we first compute that
 $T_\lambda^*f = f \circ \tau_\lambda \cdot \overline{f_\lambda \circ \tau_\lambda} \cdot \Phi_\lambda$.
 Alternatively, 
 \begin{equation}\label{eq:T-adj}
 T_\lambda^* f= \frac{\chi_{D_\lambda} \cdot (f \circ \tau_\lambda )}{f_\lambda \circ \tau_\lambda}.
 \end{equation}
Condition (CK3), and the fact that the operators $T_\lambda$ are partial isometries, now follow from straightforward calculations. 
Finally, an easy computation establishes that $T_\lambda T_\lambda^* (f) = \chi_{R_\lambda} \cdot f$ for any $\lambda \in \Lambda$, from which (CK4) follows.
}
 \end{proof}
 
 	We call the representation given in Equation \eqref{eq:T-lambda} a $\Lambda$-projective representation.

The following Proposition enables us to translate a $\Lambda$-projective system on $(X, \mu)$ to a $\Lambda$-projective system on $(X, \mu')$ for any measure $\mu'$ which is equivalent to $\mu$.
\begin{prop}
 \label{prop:lambda-proj-repn-un-equiv} Let $\Lambda$ be a finite $k$-graph with no sources. Suppose we are given a $\Lambda$-projective system  $\{\tau_{\lambda}:\lambda\in \Lambda\}$, $\{\tau^n: n \in \N^k\}$ and $\{ f_\lambda : \lambda \in \Lambda\}$ on a measure space $(X,\mu)$. Let $\mu'$ be a measure equivalent to $\mu$, and set $g_1(x)\;=\;\frac{d\mu'}{d\mu}(x)$ and $g_2(x)=\frac{d\mu}{d\mu'}(x)$.  If we define  $\{ \tilde{f}_\lambda \}_{\lambda \in \Lambda}$ by 
\begin{equation}\tilde{f}_\lambda(x)\;=\;\frac{\sqrt{g_1\circ \tau^{d(\lambda)}(x)}}{\sqrt{g_1(x)}}\cdot f_{\lambda}(x),\;\lambda\in \Lambda,
\label{eq-def-of-f-tilde}
\end{equation}
then  $\{\tau_{\lambda}:\lambda\in \Lambda\}$, $\{\tau^n: n\in \N^k\}$ and $\{ \tilde{f}_\lambda \}_{\lambda \in \Lambda}$ give a  $\Lambda$-projective system on $(X,\mu')$. 
Moreover, the associated representations $\{T_{\lambda}:\lambda\in \Lambda\}$ and $\{\tilde{T}_{\lambda}:\lambda\in \Lambda\}$ of $C^*(\Lambda)$ on $L^2(X,\mu)$ and $L^2(X,\mu')$  given by Equation  \eqref{eq:T-lambda} of Proposition \ref{prop:lambda-proj-repn} are unitarily equivalent via the unitary $U$ given by 
\begin{equation*}
U(f)(x)\;=\;\sqrt{\frac{d\mu}{d\mu'}(x)}\cdot f(x),\; f\in\; L^2(X,\mu),  \ U^{-1}(h)(x)\;=\;\sqrt{\frac{d\mu'}{d\mu}(x)}\cdot h(x),\; h\in\; L^2(X,\mu').
\end{equation*} 
\end{prop}
\begin{proof} We leave the verification of this proposition to the reader.\end{proof}

Proposition \ref{prop-2.11} below is the analog of Proposition 2.11 of \cite{dutkay-jorgensen-monic} for $\Lambda$-projective systems. 
\begin{prop}
\label{prop-2.11} Let $\Lambda$ be a finite $k$-graph with no sources. Suppose we are given two $\Lambda$-projective systems on $X$, with the same prefixing and coding maps $\{\tau_{\lambda}:\lambda\in \Lambda\}$, $\{\tau^n: n \in \N^k\}$, but with different measures $\mu, \mu'$ and $\Lambda$-projective functions $\{ f_\lambda \}_{\lambda \in \Lambda}$ for $(X, \mu)$ and $\{ f_\lambda' \}_{\lambda \in \Lambda}$ for $(X,\mu')$.

Let $d\mu' = h^2 d\mu + d\nu 
$
be the Lebesgue-Radon-Nikodym decomposition, with $h \geq 0$ and $\nu$
singular with respect to $\mu$. Then there is a partition of $X$ into Borel sets $
X = A \cup B$
such that:
\begin{enumerate}
\item[(a)]  The function $h$ is supported on $A$, $\nu$ is supported on $B$, and $\mu(B) = 0$, $\nu(A) = 0$.
\item[(b)] The sets $A$, $B$ are invariant under $\tau^n$ for all $n \in \N^k$, i.e., 
\[
(\tau^n)^{-1}(A) = A,\quad\text{and}\quad (\tau^n)^{-1}(B) = B.
\]
\item[(c)]  We have $\nu \circ  \tau_\lambda^{-1} << \nu$ and $k_\lambda := \sqrt{\frac{d(\nu \circ \tau_\lambda^{-1} )}{d\nu}}$ is supported on $B$.
\item[(d)] $| f_\lambda' | \cdot h =| f_\lambda|\cdot (h \circ \tau^{d(\lambda)} )\, {\mu\text{-a.e.~on $A$ and } |f_\lambda'| = k_\lambda\,  \nu\text{-a.e.~on } B.}$ 
\end{enumerate}
\end{prop}

\begin{proof} We start by proving (a) and (b) together. 
 Let $\tilde{B}$ be the support of $\nu$, and observe that $\mu (\tilde{B}) = 0$. We observe that the definitions of $\Lambda$-semibranching function systems and $\Lambda$-projective systems, together with the fact that $(\tau^n)^{-1}(\tilde B) = \bigcup_{\lambda \in \Lambda^n} \tau_\lambda(\tilde B)$, imply that
\[
(\tau^n)^{-1}(\tilde{B})\text{ and } (\tau_\lambda)^{-1}(\tilde{B}) 
\] 
have $\mu$-measure zero. Therefore
we can take the orbit $B$ of $\tilde{B}$ under the functions $\{\tau^n: n \in \N^k\}$ and $\{\tau_\lambda: \lambda \in \Lambda\},$ and $B$ will then have $\mu$-measure zero. Let
$A := X \backslash B$. Then $A$ contains the support of $\mu$, and we can choose $h$ to be supported on $A$.  Moreover, 
$\nu(A) = 0$. By construction, $A$ and $B$ are invariant under $\tau^n$. This establishes (a) and (b).
To prove (c), let $E$ be a Borel set with $\nu (E) = 0$. Then $\nu (E\cap B) = 0$, so the fact that $\mu$ vanishes on $B$ implies that $\mu' (E\cap B) = 0$. 
We consequently have $\mu'(\tau_\lambda^{-1} (E \cap B)) = 0,$ which means that  $\mu'(\tau_\lambda^{-1} (E) \cap  B) = 0$, so $\nu(\tau_\lambda^{-1} (E)) = 0$.
Since $B$ is invariant under $\tau_\lambda^{-1}$ and $\nu$ and $\nu \circ \tau_\lambda^{-1}$  are supported on $B$, it follows that 
$k_\lambda$ is supported on $B$. 
To see (d), let $f$ be a bounded Borel function supported on $A$. Then we have
\[
\begin{split}
\int_A |f'_\lambda|^2\, f\, h^2\, d\mu &=\int_A |f'_\lambda|^2\, f\,  d\mu'= \int_A f \, \frac{d(\mu' \circ \tau_\lambda^{-1})}{d(\mu')}d\mu'  =\int_A (f\circ \tau_\lambda)\,  d\mu' \\
&= \int_A (f\circ \tau_\lambda)\, h^2\, d\mu = \int_X (f\circ \tau_\lambda)\, (h^2 \circ \tau^{d(\lambda)}\circ \tau_\lambda)\, d\mu= \int_X f\, (h^2 \circ \tau^{d(\lambda)})\, d(\mu\circ \tau_\lambda^{-1})\\
&= \int_X f\, (h^2 \circ \tau^{d(\lambda)})\, | f_\lambda|^2 d\mu,\\
\end{split}
\]
which implies the first relation. The second relation follows from the fact that
$ \mu'|_B = \nu$.
\end{proof}

  \subsection{Projection valued measures}
\label{sec:Proj-valued-measures}
The second technical tool which underpins our analysis of the monic representations of $C^*(\Lambda)$ is  the   projection valued measure associated to a  representation of $C^*(\Lambda)$. 
Our work in this section is inspired by Dutkay, Haussermann, and Jorgensen     \cite{dutkay-jorgensen-monic, dutkay-jorgensen-atomic}. 

\begin{defn} 
	\label{def-proj-val-measu}
Let $\Lambda$ be a finite $k$-graph with no sources.
Given a representation $\{ t_\lambda\}_{\lambda\in\Lambda }$ of a $k$-graph $C^*$-algebra $C^*(\Lambda)$ on a Hilbert space $\mathcal{H}$, we define a projection valued function $P$ on $\Lambda^\infty$ by 
 \[P(Z(\lambda)) = t_\lambda t_\lambda^* \quad\text{for all $\lambda \in\Lambda$}.
 \]
 \end{defn}
 


In the proof (Proposition \ref{conj-palle-proj-valued-measure-gen-case})
that $P$ indeed defines a projection-valued measure on $\Lambda^\infty$, we rely on the following well-known Lemma. Thus, in our application, $X = \Lambda^\infty$ and 
$ \mathcal F_n$ will be the $\sigma$-algebra generated by the cylinder sets $Z(\lambda)$ with $d(\lambda) = (n, \ldots, n)$.

\begin{lemma}[Kolmogorov Extension Theorem, \cite{kolmogorov, Tum}]
\label{lem-Kolm}
Let $(X, \mathcal{F}_n, \nu_n)_{n\in\N}$ be a sequence of probability measures $(\nu_n)_{n\in \N}$ on the same space $X$, each associated with a $\sigma$-algebra $\mathcal F_n$
; further assume that  $(X, \mathcal{F}_n, \nu_n)_{n\in\N}$  form a projective system, i.e., an inverse limit. Suppose that Kolmogorov's consistency condition holds:
\[ \nu_{n+1}|_{\mathcal F_n} = \nu_n.\]
Then there is a unique extension $\nu$ of the measures $(\nu_n)_{n\in \N}$ to the $\sigma$-algebra $\bigvee_{n\in \N} \mathcal F_n$ generated by $\bigcup_{n\in \N} \mathcal F_n$.
\label{lem:kolmogorov}
\end{lemma}

 
 \begin{prop} 
 \label{conj-palle-proj-valued-measure-gen-case}
 Let $\Lambda$ be a finite $k$-graph with no sources.
 Given a representation $\{ t_\lambda\}_{\lambda\in\Lambda}$ of a $k$-graph $C^*$-algebra $C^*(\Lambda)$ on a Hilbert space $\mathcal{H}$,  the function $P$ of Definition \ref{def-proj-val-measu}
 extends to  a projection valued measure on the Borel $\sigma$-algebra  $\mathcal{B}_o(\Lambda^\infty)$ of the infinite path space $\Lambda^\infty$. 
 \end{prop}

 \begin{proof}
 {Recall from the proof of \cite[Lemma 4.1]{FGKP}  that 
 \[\{Z(\lambda): d(\lambda) = (n, n, \ldots, n) \text{ for some } n \in \N\}\]
  generates the topology on $\Lambda^\infty$.  Thus, 
  $\mathcal B_o(\Lambda^\infty) = \varinjlim \mathcal F_n.$  By Lemma \ref{lem-Kolm}, it therefore suffices to show that 
 \[ P(Z(\lambda)) = \sum_{\eta \in s(\lambda) \Lambda^{(1,\ldots, 1)}} P(Z(\lambda \eta))\]
 whenever $d(\lambda) = (n, \ldots, n)$ for some $n \in \N$.  However, this follows immediately from (CK4):
 \begin{equation} P(Z(\lambda)) = t_\lambda t_\lambda^* = t_\lambda\left( \sum_{\eta \in s(\lambda) \Lambda^{(1,\ldots, 1)}}  t_\eta t_\eta^*\right)  t_\lambda^* = \sum_{\eta \in s(\lambda) \Lambda^{(1,\ldots, 1)}} P(Z(\lambda \eta).
\label{eq:P-additivity} 
 \end{equation}}
 \end{proof}

We now record some properties of  $P$ which we will rely on in the sequel.
 The equations below are the analogues for $k$-graphs of Equations (2.7) and (2.8) and (2.13) of \cite{dutkay-jorgensen-atomic}.

 \begin{prop} 
 \label{prop-atomic-basic-equns} Let $\Lambda$ be a row-finite, source-free $k$-graph, and fix a representation $\{t_\lambda:\lambda \in \Lambda\}$ of $C^*(\Lambda)$.
 \begin{itemize}\label{prop:pvm-properties}
 \item[(a)] For $\lambda, \eta \in\Lambda$ with $s(\lambda)=r(\eta)$, we have $t_\lambda P(Z(\eta)) t_\lambda^*=P(\sigma_\lambda(Z(\eta)))$, where $\sigma_\lambda$ is the {prefixing} 
  map on $\Lambda^\infty$ given in Equation \eqref{eq:shift-map}.
 
 \item[(b)] For any fixed $n \in \N^k$, we have
  \[
  \sum_{\lambda \in  f(\eta) \Lambda^n} t_\lambda P(\sigma_\lambda^{-1}(Z(\eta))) t_\lambda^* = P(Z(\eta)) ;
  \]
  \item[(c)]  For any $\lambda,\eta \in \Lambda$ with $r(\lambda)=r(\eta)$, we have $t_\lambda P(\sigma_\lambda^{-1}(Z(\eta))) = P(Z(\eta)) t_\lambda$;
  \item[(d)] When $\lambda\in\Lambda^n$, we have $t_\lambda P( Z(\eta)) =  P((\sigma^n)^{-1}(Z(\eta))) t_\lambda$.
 \end{itemize}
 \end{prop}

 \begin{proof}
Straightforward calculation.
 \end{proof}

 \subsection{Disjoint and irreducible representations}
 \label{sub-disj-irred}
 
In this section we will derive from the technical results in Section \ref{sec:lambda-proj-systems}
important consequences that detail 
when representations of $k$-graph $C^*$-algebras are disjoint or irreducible.   In particular, Theorem \ref{thm-disjoint-monic-repres} suggests the importance of dealing with $\Lambda$-projective systems with non-negative functions $f_\lambda$.  We will focus more exclusively on such $\Lambda$-projective systems in Section \ref{sec:univ_repn} below.

 \begin{thm} (C.f.~Theorem 2.12 of \cite{dutkay-jorgensen-monic})
 \label{thm-disjoint-monic-repres}
Let $\Lambda$ be a finite  $k$-graph with no sources. Suppose we are given two $\Lambda$-projective systems on the infinite path space $\Lambda^\infty$ with the standard prefixing and coding maps $\{\sigma_{\lambda}:\lambda\in \Lambda\}$, $\{\sigma^n: n \in \N^k\}$, but associated to different measures $\mu, \mu'$ and different $\Lambda$-projective families of non-negative functions $\{ f_\lambda \}_{\lambda \in \Lambda}$ on $(\Lambda^\infty,\mu)$, and $\{ f_\lambda' \}_{\lambda \in \Lambda}$ on $(\Lambda^\infty,\mu')$. Then the two associated representations $\{T_\lambda: \lambda \in \Lambda\}$ and $\{T_\lambda': \lambda \in \Lambda\}$  of $C^*(\Lambda)$ given by Equation \eqref{eq:T-lambda} of Proposition  \ref{prop:lambda-proj-repn}  are disjoint if and only if the measures $\mu$ and $\mu'$ are mutually singular.
 \end{thm}

 \begin{proof}

If the representations are not disjoint, there exist subspaces $\H_\mu \subseteq L^2(\Lambda^\infty, \mu)$ and $\H_{\mu'} \subseteq L^2(\Lambda^\infty, \mu')$, preserved by their respective representations, and a unitary $W: \H_\mu \to \H_{\mu'}$ such that 
\[W T_\lambda |_{\H_\mu} = T_\lambda' |_{\H_{\mu'}} W, \qquad W T_\lambda^* |_{\H_\mu} = (T_\lambda')^* |_{\H_{\mu'}} W.\]
The fact that each operator $T_\lambda^*$ also preserves $\H_\mu$ implies that 
\begin{align*}
W T_\lambda T_\lambda^* |_{\H_\mu} &= W T_\lambda |_{\H_\mu} T_\lambda^* |_{\H_\mu} = T_\lambda' |_{\H_{\mu'}} W T_\lambda^* |_{\H_\mu} \\
&= T_\lambda' (T_\lambda')^*|_{\H_{\mu'}} W.
\end{align*}
Moreover, it follows easily from the formulas for $T_\lambda$ and $T_\lambda^*$ in Equations \eqref{eq:T-lambda} and \eqref{eq:T-adj} that 
\[ T_\lambda T_\lambda^* = M_{\chi_{Z(\lambda)}} = T_\lambda' (T_\lambda')^*.\]
In other words,  the representations of $C(\Lambda^\infty)$ given by $\chi_{Z(\lambda)} \mapsto T_\lambda T_\lambda^*$ and $\chi_{Z(\lambda)} \mapsto T_\lambda' (T_\lambda')^*$ (on $L^2(\Lambda^\infty, \mu)$ and $L^2(\Lambda^\infty, \mu')$ respectively) are multiplication representations.
Since $W$ implements a unitary equivalence between their subrepresentations on $\H_\mu$ and $H_{\mu'}$ respectively, Theorem 2.2.2 of \cite{arveson} tells us that  the measures $\mu, \mu'$ cannot be mutually singular.  
 
 For the converse, assume that the representations are disjoint and that the measures
 are not mutually singular. Then, use Proposition \ref{prop-2.11}  and decompose $ d\mu' = h^2 d\mu +d\nu$,
 with the subsets $A$, $B$ as in Proposition \ref{prop-2.11}.
 Define the operator $W$ on $L^2(\Lambda^\infty, \mu')$ by $W(f) = f\cdot h$ if $f \in  L^2(A, \mu')$, and $W(f) = 0$ on the
 orthogonal complement of $L^2(A, \mu') \subseteq L^2(\Lambda^\infty, \mu')$. Since $A$ is invariant under $\tau^n$ for all $ n$, $L^2(A, \mu')$ is an invariant subspace
 for the representation. To  check that $W$ is intertwining, we use part (d) of Proposition \ref{prop-2.11} and the non-negativity condition on $\{f_\lambda\}$ and $\{f'_\lambda\}$ to obtain the almost-everywhere equalities
 \[
 T_\lambda W (f )= f_\lambda(h \circ \tau^{d(\lambda)})(f \circ \tau^{d(\lambda)}) = f_\lambda' \, h\, (f \circ \tau^{d(\lambda)}) = W T_\lambda'(f).
 \]
 Since $W$ intertwines the representations $\{ T_\lambda\}_{\lambda \in \Lambda}, \{ T_{\lambda}'\}_{\lambda \in \Lambda}$ of $C^*(\Lambda)$, we must have $W =0$; hence $h=0$, so $\mu, \mu'$ are mutually singular.
 \end{proof}
 
 \begin{rmk}
 As a Corollary of  Theorem  \ref{thm-disjoint-monic-repres}, 
 we see  that the examples of Markov measures introduced in 
\cite[Section 4.2]{FGJorKP}
 generate representations of $C^*(\Lambda)$ disjoint from the  representation of  \cite[Theorem 3.5]{FGKP};  {see also Example \ref{ex:markov-is-monic} below.} In fact, these  Markov measures are 
 mutually singular with the Perron--Frobenius measure \cite{dutkay-jorgensen-monic}, \cite{Kaku}.
 \end{rmk}

%

  \begin{thm} (C.f. Theorem 2.13 of
   \label{thm-irred-monic-repres} \cite{dutkay-jorgensen-monic})\label{thm:ergodic}
    Let $\Lambda$ be a finite $k$-graph with no sources.  
  Suppose that the infinite path space $\Lambda^\infty$ admits a $\Lambda$-projective system on $(\Lambda^\infty,\mu)$ for some measure $\mu$, with  the standard prefixing maps $\{\sigma_\lambda\, :\, \lambda\in\Lambda\}$ and coding maps $\{\sigma^n\, :\, n\in\N^k\}$  of Definition \ref{def:infinite-path}.  
    Let $\{T_\lambda:\lambda\in\Lambda\}$ be the associated representation of $C^*(\Lambda)$. 
 Then:
 \begin{itemize}
 \item[(a)] The commutant of the operators $\{T_\lambda: \lambda \in \Lambda\}$ consists of multiplication
      operators by functions $h$ with $h \circ \sigma^n = h$, $\mu$-a.e for all $n\in \N^k$.  
 \item[(b)] The representation given by $\{T_\lambda:\lambda\in\Lambda\}$ is
      irreducible if and only if the coding maps $\sigma^n$ are jointly ergodic with respect to the measure $\mu$, i.e., the only Borel
      sets $A\subset \Lambda^\infty$  with $ (\sigma^{n})^{-1}(A)= A$ for all $n$ are sets of measure zero, or of full measure.       \end{itemize}
       
  \end{thm}

 \begin{proof}
We first observe that the commutant of $\{T_\lambda\}_{\lambda \in \Lambda}$ is contained in $C(\Lambda^\infty)' \subseteq B(L^2(\Lambda^\infty, \mu))$ and hence consists of multiplication operators. The proof of part (a) is then a straightforward calculation, and part (b) follows from part (a) and the definition of ergodicity.
 \end{proof}

\section{Monic representations of finite $k$-graph algebras}
\label{sec:monic-results} 

The first main result of this section, Theorem \ref{thm-characterization-monic-repres}, establishes that every monic  representation 
of a finite, strongly connected $k$-graph algebra $C^*(\Lambda)$ is unitarily equivalent to  a $\Lambda$-projective representation of $C^*(\Lambda)$ on $L^2(\Lambda^\infty, \mu_\pi)$, where the measure $\mu_\pi$ arises from the representation.  (See Definition \ref{def:lambda-monic-repres} and  Equation \eqref{eq:monic_measure} below for details.)  After proving Theorem \ref{thm-characterization-monic-repres}, we 
 examine a variety of  examples of representations of $C^*(\Lambda)$, and identify which representations are monic.  This analysis requires our second main result, Theorem \ref{prop-conv-to-8-7}, which provides a measure-theoretic characterization of when a $\Lambda$-semibranching representation is monic. 


\begin{defn}\label{def:lambda-monic-repres}
 Let $\Lambda$ be a finite $k$-graph with no sources.
  A representation $\{ t_\lambda\,:\, \lambda\in\Lambda\}$ of a $k$-graph  on a Hilbert space $\H$ is called \emph{monic} if $t_\lambda \not= 0$ for all $\lambda \in \Lambda$, and there exists a vector $\xi\in \mathcal{H}$ such that
 \[
 \overline{\text{span}}_{\lambda \in \Lambda} \{ t_\lambda t_\lambda^* \xi \} = \mathcal{H}.
 \]
 \end{defn}

From the projection valued measure $P$ 
associated to $\{t_\lambda:\lambda \in \Lambda\}$ as in Theorem~\ref{conj-palle-proj-valued-measure-gen-case}, we obtain a representation $\pi:C(\Lambda^\infty)\to B( \mathcal{H})$: 
\[
\pi(f)=\int_{\Lambda^\infty}f(x)dP(x),
\]
which gives, for $\lambda\in\Lambda$,
\begin{equation}\label{eq:repn_pi}
\pi(\chi_{Z(\lambda)})=\int_{\Lambda^\infty}\chi_{Z(\lambda)}\,(x)\,dP(x)=P(Z(\lambda))=t_\lambda t^*_\lambda.
\end{equation}
Since 
we can view $C(\Lambda^\infty)$ as a subalgebra of $C^*(\Lambda)$ via the embedding $\chi_{Z(\lambda)} \mapsto t_\lambda t_\lambda^*$, the representation $\pi$ is often understood as  the restriction of the representation $\{ t_\lambda \}_{\lambda \in \Lambda}$ to the ``diagonal subalgebra'' $\overline{\text{span}}\{ t_\lambda t_\lambda^*\}_{\lambda \in \Lambda}.$

If the representation $\{t_\lambda\}_\lambda$ is monic, 
there is a cyclic vector $\xi\in \mathcal{H}$ for $\pi$.  This induces} 
we obtain  a Borel measure $\mu_\pi$ on $\Lambda^\infty$ given by
\begin{equation}\label{eq:monic_measure}
\mu_\pi(Z(\lambda))=\langle \xi, P(Z(\lambda))\xi\rangle=\langle \xi, t_\lambda t^*_\lambda\xi \rangle.
\end{equation}

\begin{thm}
\label{thm-characterization-monic-repres} 
Let $\Lambda$ be a finite  $k$-graph with no sources.  
If $\{t_\lambda\}_{\lambda \in \Lambda}$ is a monic representation of $C^*(\Lambda)$ on a Hilbert space $\H$,
then $\{t_\lambda\}_{\lambda\in \Lambda}$ is unitarily equivalent to a representation $\{S_\lambda\}_{\lambda\in \Lambda}$ associated to a $\Lambda$-projective system on $(\Lambda^\infty, \mu_\pi)$, {which is associated to the standard coding and prefixing maps $\sigma^n, \sigma_\lambda$ of Definition \ref{def:infinite-path}}.

Conversely, if we have a representation of $C^*(\Lambda)$ on $L^2(\Lambda^\infty, \mu)$ which arises from a $\Lambda$-projective system {associated to the standard coding and prefixing maps $\sigma^n, \sigma_\lambda$}, then the representation is monic. 
\end{thm}
 
 By Example \ref{ex:standard-SBFS-is-proj}, this implies that a $\Lambda$-semibranching function system on $(\Lambda^\infty, \mu)$, for any Borel measure $\mu$, gives rise to a monic representation of $C^*(\Lambda)$.

\begin{proof}
 Suppose that  the representation $\{t_\lambda\}_{\lambda \in \Lambda}$ of $C^*(\Lambda)$ is monic, and let $\xi\in \H$ be a cyclic vector for $C(\Lambda^\infty)$.
 Note that the map $W : C(\Lambda^\infty) \to  \mathcal{H}$ given by 
\[
  W( f) = \pi(f) \xi
\]
is 
 linear. Moreover, if we think of $C(\Lambda^\infty)$ as a dense subspace of $L^2(\Lambda^\infty, \mu_\pi)$, the operator $W$ is isometric:
\[ 
 \| f\|^2_{L^2} = \int_{\Lambda^\infty} |f|^2 \, d\mu_\pi  = \langle \xi, \pi(|f|^2) \xi \rangle = \| \pi(f) \xi \|^2 = \|W(f)\|^2.
\]
Therefore $W$ extends to an isometry  from $
L^2(\Lambda^\infty, \mu_\pi) $ to $ \mathcal{H}.$
Since $W$ is also onto (because the representation is monic), $W$ is a surjective isometry; that is, $W$ is a unitary.

Moreover, for any $f \in C(\Lambda^\infty)$ and any $\varphi \in L^2(\Lambda^\infty, \mu_\pi)$, we have 
\[ \pi(f) W(\varphi) = \pi(f) \pi(\varphi) \xi = \pi(f \cdot \varphi) \xi = W(f \cdot \varphi).\]
Thus, unitarity of  $W$ implies that $W^* \pi(f) W$ acts on $L^2(\Lambda^\infty, \mu_\pi)$ by multiplication by $f$:
\begin{equation}
W^* \pi(f) W = M_f \text{ and } W M_f W^* = \pi(f).\label{eq:w-intertwines-pi}
\end{equation}

 Now define the operator
$
S_\lambda = W^* t_\lambda W \text{ for $\lambda \in \Lambda$.} 
$  By construction, 
 the operators $\{S_\lambda\}_{\lambda \in \Lambda}$ also  give  a representation of $C^*(\Lambda)$. Moreover, since $W$ is a unitary, 
\begin{equation}\begin{split} S_\lambda S_\lambda^* (f) &= W^* t_\lambda t_\lambda^*  W(f) = W^* \pi(\chi_{Z(\lambda)}) \pi(f) \xi = W^* \pi( \chi_{Z(\lambda)} \cdot f) \xi \\
&= W^* W (\chi_{Z(\lambda)} \cdot f) = \chi_{Z(\lambda)} \cdot f.
\end{split}
\label{eq:range-sets}
\end{equation}

Let $1$ denote the characteristic function of $\Lambda^\infty$, and define a function $f_\lambda \in L^2(\Lambda^\infty, \mu_\pi)$ by 
\[
f_\lambda = S_\lambda 1 = W^* t_\lambda \xi.
\]
 We will now show that the functions $f_\lambda$, combined with the usual coding and prefixing maps $\{\sigma^n, \sigma_\lambda\}_{n, \lambda}$ on $\Lambda^\infty$, form a $\Lambda$-projective system on $(\Lambda^\infty, \mathcal{B}_o(\Lambda^\infty), \mu_\pi)$. 
  To that end, we will invoke Proposition \ref{prop:pvm-properties}.  Since
\[ P(Z(\nu)) = \pi(\chi_{Z(\nu)})\]
for any $\nu \in \Lambda$, and the proof of \cite[Lemma 4.1]{FGKP} shows that characteristic functions of cylinder sets densely span $L^2(\Lambda^\infty, \mu_\pi)$, the equalities established in Proposition \ref{prop:pvm-properties} still hold if we replace $P(Z(\nu))$ by $\pi(f)$ for any $f \in L^2(\Lambda^\infty, \mu_\pi)$.  In particular, noting that 
\[ \chi_{(\sigma^n)^{-1}(Z(\nu))} = \chi_{Z(\nu)} \circ \sigma^n \quad \text{ and } \quad \chi_{\sigma_\lambda^{-1}(Z(\nu))} = \chi_{Z(\nu)} \circ \sigma_\lambda \]
Part (d) of Proposition \ref{prop:pvm-properties} implies that if $d(\lambda) = n$, 
\begin{equation}
t_\lambda \pi(f) = \pi(f \circ \sigma^n) t_\lambda\label{eq:pvm-useful}
\end{equation} 
and Part (c) implies that 
\begin{equation}
t_\lambda^* \pi(f) = \pi(f \circ \sigma_\lambda) t_\lambda^*.\label{eq:part-c-useful}
\end{equation}

Let $f \in L^2(\Lambda^\infty, \mu_\pi)$ and let $n = d(\lambda)$.  By using Part (d) of Proposition \ref{prop:pvm-properties}, Equation \eqref{eq:w-intertwines-pi}, and the fact that $W$ is a unitary, we obtain 
\begin{align*}
S_\lambda (f) &= W^* t_\lambda W ( f )= W^* t_\lambda \pi(f)\xi \\
&= W^* \pi(f\circ \sigma^n) t_\lambda \xi =  W^* \pi(f\circ \sigma^n)W W^* t_\lambda \xi\\
&=(f\circ \sigma^n) \cdot  f_\lambda.
\end{align*}

In order to show that $\{S_\lambda\}_{\lambda \in \Lambda}$ is a $\Lambda$-projective representation, then, Proposition \ref{prop:lambda-proj-repn} tells us that it remains to check that the standard prefixing and coding maps make $(\Lambda^\infty, \mu_\pi)$ into a $\Lambda$-semibranching function system, and that Condition (a) of Definition \ref{def:lambda-proj-system} holds for the functions $f_\lambda$.

To establish Condition (a), we work indirectly.
Since $W$ is a unitary, we   have (for  any $f \in L^2(\Lambda^\infty, \mu_\pi)$ and any $\lambda \in \Lambda^n$)
\begin{align*}
\int_{\Lambda^\infty} |f_\lambda|^2 \cdot  f \, d\mu_\pi &= \langle S_\lambda(1), S_\lambda(1) \cdot  f \rangle_{ L^2} = \langle W^* t_\lambda (\xi), M_f W^* t_\lambda (\xi)\rangle_{ L^2} \\
&= \langle t_\lambda \xi,  W M_f W^*(t_\lambda \xi )  \rangle_{ \mathcal{H}} = \langle \xi, t_\lambda^* \pi(f) t_\lambda \xi \rangle_{\mathcal{H}} = \langle \xi, \pi(f \circ \sigma_\lambda) \xi \rangle_{\mathcal{H}}\\
&  = \int_{\Lambda^\infty} f \circ \sigma_\lambda \, d\mu_\pi = \int_{\Lambda^\infty} f \, d(\mu_\pi \circ \sigma_\lambda^{-1}).
\end{align*}

If $E\subseteq \Lambda^\infty$ is any set for which $\mu_\pi(E) = 0$, then taking $f = \chi_E$ above  shows that $\mu_\pi \circ \sigma_\lambda^{-1}(E) = 0$ also -- in other words, 
\begin{equation}
\mu_\pi \circ \sigma_\lambda^{-1} << \mu_\pi.\label{eq:mu-pi-abs-cts}
\end{equation}
The uniqueness of Radon-Nikodym derivatives then implies that
\[|f_\lambda|^2 =  \frac{d(\mu_\pi \circ \sigma_\lambda^{-1})}{d(\mu_\pi )}. \]
In other words, Condition (a) of Definition \ref{def:lambda-proj-system} holds.

{We now show that $f_\lambda \not= 0$ a.e.~on $Z(\lambda)$.  Define $E_\lambda \subseteq Z(\lambda)$ by 
\[ E_\lambda := \{ x \in Z(\lambda): f_\lambda(x) = 0\}.\]
Then 
$0 = \int_{E_\lambda} |f_\lambda|^2 \, d\mu_\pi = \mu_\pi \circ \sigma_\lambda^{-1}(E_\lambda) = \pi(\chi_{\sigma_\lambda^{-1} (E_\lambda)}) = t_\lambda^* \pi(\chi_{E_\lambda}) t_\lambda$ by Proposition \ref{prop:pvm-properties}.

By hypothesis, $t_\lambda \not= 0$, so there exists $\zeta \in \H$ such that $t_\lambda(\zeta)\not= 0$.  However, for any $\zeta$,
\[ \langle t_\lambda(\zeta), \pi(\chi_{E_\lambda}) t_\lambda(\zeta) \rangle = \langle t_\lambda(\zeta), \pi (\chi_{E_\lambda})^2 \pi(\chi_{Z(\lambda)}) t_\lambda(\zeta) \rangle = \langle t_\lambda^* \pi(\chi_{E_\lambda}) t_\lambda (\zeta), t_\lambda^* \pi(\chi_{E_\lambda}) t_\lambda(\zeta) \rangle = 0\]
by the Cuntz--Krieger relations and the fact that $\pi(C(\Lambda^\infty))$ is abelian.  In other words, $\pi(\chi_{E_\lambda})$ is orthogonal to the range projection $\pi(\chi_{Z(\lambda)})$ of $t_\lambda$.

On the other hand, $\chi_{E_\lambda} \chi_{Z(\lambda)} = \chi_{E_\lambda}$ since $E_\lambda \subseteq Z(\lambda)$.  It follows that $\pi(\chi_{E_\lambda}) = 0$; equivalently, $\mu_\pi(E_\lambda) = 0$.  In other words, the set $E_\lambda \subseteq Z(\lambda)$ of points where $f_\lambda = 0$ has $\mu_\pi$-measure zero, as claimed.

}

Similarly, for any set $F \subseteq Z(s(\lambda))$ such that  $\mu_\pi(F) = 0$, taking $f = \chi_{\sigma_\lambda(F)}$ reveals that 
\[ 0 = \mu_\pi (F) = \mu_\pi \circ \sigma_\lambda^{-1}(\sigma_\lambda(F)) = \int_{\sigma_\lambda(F)} |f_\lambda|^2 \, d\mu_\pi.\]
Since  $|f_\lambda|^2 >0$ a.e.~on $Z(\lambda) \supseteq \sigma_\lambda(F)$, we must have $ \mu_\pi \circ \sigma_\lambda(F) = 0$
and hence $\mu_\pi \circ \sigma_\lambda << \mu_\pi$.
Furthermore, the Radon-Nikodym derivative $\frac{d(\mu_\pi \circ \sigma_\lambda)}{d(\mu_\pi )}$ is nonzero $\mu_\pi$-a.e.~on $Z(s(\lambda))$. 
To see this, we  set  
\[ E = \left\{ x \in Z(s(\lambda)): \frac{d(\mu_\pi \circ \sigma_\lambda)}{d(\mu_\pi )} = 0 \right\}\]
and observe that 
\[\mu_\pi (\sigma_\lambda(E)) = \int_E \frac{d(\mu_\pi \circ \sigma_\lambda)}{d(\mu_\pi )} \, d\mu_\pi = 0.\]
Equation \eqref{eq:mu-pi-abs-cts} 
therefore  implies that
$\mu_\pi (E) = (\mu_\pi \circ \sigma_\lambda^{-1})(\sigma_\lambda(E)) = 0.$

\cite[Proposition 4.1]{FGJorKP} now implies that the standard prefixing and coding maps make $(\Lambda^\infty, \mu)$ into a $\Lambda$-semibranching function system.  Consequently, the functions $f_\lambda$ make $\{S_\lambda\}_{\lambda \in \Lambda}$ into a $\Lambda$-projective representation, which is unitarily equivalent to our initial monic representation by construction. 

For the converse, suppose that $\{t_\lambda\}_{\lambda\in\Lambda}$ is a representation of $C^*(\Lambda)$ on $L^2(\Lambda^\infty, \mu)$, for some Borel measure $\mu$, which arises from a $\Lambda$-projective system $\{f_\lambda\}_{\lambda\in\Lambda}$   {associated to the standard coding and prefixing maps $\{ \sigma^n, \sigma_\lambda\}_{n, \lambda}$. The computations from Proposition \ref{prop:lambda-proj-repn} establish that $t_\lambda t_\lambda^*$ is given by multiplication by $\chi_{R_\lambda} = \chi_{Z(\lambda)}$.} 
%
Consequently, $1 = \chi_{\Lambda^\infty}$ is a cyclic vector for $C(\Lambda^\infty) \subseteq C^*(\Lambda)$.  Thus, $\{t_\lambda\}_{\lambda \in \Lambda}$ is monic.
\end{proof}


\begin{rmk}
\label{rmk:strengthening-monic}
 In the final section of their paper \cite{bezuglyi-jorgensen}, Bezuglyi and Jorgensen studied the relationship between semibranching function systems and monic representations of Cuntz--Krieger algebras (1-graph $C^*$-algebras). 
Theorem 5.6 of \cite{bezuglyi-jorgensen} establishes that within a specific class of semibranching function systems, which the authors term \emph{monic systems}, those for which the underlying space is the infinite path space $\Lambda^\infty$ are precisely the systems which give rise to monic representations of the Cuntz--Krieger algebra.  The $\Lambda$-projective systems studied in Section \ref{sec:lambda-proj-systems} constitute our extension to $k$-graphs of the monic systems for Cuntz--Krieger algebras.
Thus, even in the case of 1-graph algebras (Cuntz--Krieger algebras), our Theorem \ref{thm-characterization-monic-repres} is substantially stronger than Theorem 5.6 of \cite{bezuglyi-jorgensen}: our Theorem \ref{thm-characterization-monic-repres} gives a complete characterization of monic representations, without the hypothesis that such representations arise from a monic or $\Lambda$-projective system.
\end{rmk}

  
\begin{thm} 
\label{thm-theorem-2.9-of-monic}
  Let $\Lambda$ be a finite, source-free $k$-graph, and let $\{S_\lambda\}_{\lambda \in \Lambda}$, $\{T_\lambda\}_{\lambda \in \Lambda}$ be two monic representations   of $C^*(\Lambda)$. Let $\mu_S, \mu_T$ be the measures on $\Lambda^\infty$ associated to these representations as in \eqref{eq:monic_measure}.
  The representations $\{S_\lambda\}_{\lambda \in \Lambda}$, $\{T_\lambda\}_{\lambda \in \Lambda}$ are  equivalent if and only if the measures  $\mu_S$ and  $\mu_T$ are equivalent 
  and there exists a function $h$ on $\Lambda^\infty$ such that
  \begin{equation}\label{eq-1-thm-monic}
  \frac{d\mu_S}{d\mu_T} =|h|^2 \quad \text{and}
  \end{equation}
 \begin{equation}\label{eq-2-thm-monic}
  f^S_{\lambda} = \frac{ h \circ \sigma^n}{h }f^T_{\lambda} \quad\text{for all  $\lambda\in\Lambda$ with $d(\lambda)=n$.}
  \end{equation}
 \end{thm}
    
  \begin{proof} 
  Suppose $\{S_\lambda\}_{\lambda \in \Lambda}$, $\{T_\lambda\}_{\lambda \in \Lambda}$ are equivalent representations of $C^*(\Lambda)$.  {From Theorem 
  \ref{thm-disjoint-monic-repres}, it follows that the associated measures $\mu_S, \mu_T$ are equivalent.}
  Let
  $
  W:L^2(\Lambda^\infty, \mu_S) \to L^2(\Lambda^\infty,\mu_T)
  $
  be the intertwining 
  unitary for them. Then the two representations are also equivalent when restricted to  the diagonal  subalgebra $C^*(\{ t_\lambda t_\lambda^*: \lambda \in \Lambda \})$.
By linearity, we can extend the formula from Equation \eqref{eq:range-sets} to all of $C(\Lambda^\infty)$.  It follows that $\pi_S, \pi_T$ are both given on $C(\Lambda^\infty)$ by multiplication: 
  \[\pi_S(\phi) = M_\phi \text{ and } \pi_T(\phi) = M_\phi \ \forall \ \phi \in C(\Lambda^\infty).\]
Since $W$ intertwines a dense subalgebra -- namely $\pi_S(C(\Lambda^\infty))$ -- of the maximal abelian subalgebra $L^\infty(\Lambda^\infty, \mu_S) \subseteq B(L^2(\Lambda^\infty, \mu_S))$ which consists of multiplication operators,  with the dense subalgebra $\pi_T(C(\Lambda^\infty)) \subseteq L^\infty(\Lambda^\infty, \mu_T)$, the unitary $W$ must be given by multiplication by some nowhere-vanishing function $h$ on $\Lambda^\infty$: 
   $W(f) = h f.$
   Moreover, since $W$ is a unitary, 
  \[
 \int_{\Lambda^\infty} |W(f)|^2 d\mu_T =  \int_{\Lambda^\infty} |f|^2|h|^2 d\mu_T = \int_{\Lambda^\infty}|f|^2 d\mu_S \quad\text{for all  $f \in L^2(\Lambda^\infty,\mu_S)$},
  \]
  which implies \eqref{eq-1-thm-monic}.  
   
  From the intertwining property   $
  T_\lambda \, W = W \, S_{\lambda} $ 
  we obtain, for any $f \in L^2(\Lambda^\infty, \mu_S)$ and  any $\lambda$ with $d(\lambda)=n$, that 
  \[
   T_\lambda \, W (f) = W \, S_{\lambda}(f) ,\ \text{that is, }
  f^T_\lambda(h \circ \sigma^n) (f \circ \sigma^n) = hf^S_\lambda   (f \circ \sigma^n). 
  \]
  Take $f=1$ and we obtain that 
  \[
  f^T_\lambda \frac{h \circ \sigma^n}{h} = f^S_\lambda\label{eq:10-1}
  \]
  as claimed in \eqref{eq-2-thm-monic}.

  For the converse, suppose that the measures $\mu_S, \mu_T$ are equivalent and there is a function $h$ on $\Lambda^\infty$ satisfying \eqref{eq-1-thm-monic} and \eqref{eq-2-thm-monic}.  Then  define $W: L^2(\Lambda^\infty, \mu_S) \to L^2(\Lambda^\infty, \mu_T)$ by 
  \[
  W f=hf ; \] 
 it is then straightforward to check that $W S_\lambda = T_\lambda W$ and that $W$ is a unitary.\end{proof} 
 

\subsection{$\Lambda$-semibranching function systems and monic representations}
\label{sec:monic-examples}

In this section, we discuss several  examples of $\Lambda$-semibranching function systems and identify which of them give rise to monic representations of $C^*(\Lambda)$ -- or, equivalently, which  are unitarily equivalent to $\Lambda$-semibranching function systems on the infinite path space. 
First, we provide another characterization of monic representations.  The next theorem shows that a $\Lambda$-semibranching system on $(X, \mu)$ induces a monic  representation of $C^*(\Lambda)$ if and only if its associated range sets generate the $\sigma$-algebra of $X$. To  state our result more precisely, we will  
denote by
 $
 ({X,} \mathcal F, \mu)
 $
the measure space associated to $L^2(X, \mu)$; in particular, $\mathcal F$ is the standard $\sigma $-algebra associated to $L^2(X, \mu)$.

\begin{thm}
\label{prop:range-sets-generate-sigma-alg-in-monic}
Let $\Lambda$ be a finite, source-free $k$-graph and 	let  $ \{t_\lambda\}_{\lambda \in \Lambda}$ be a $\Lambda$-semibranching  representation of $C^*(\Lambda)$ on $L^2(X, \mathcal{F}, \mu)$ 
with $\mu(X)< \infty$. Let   $\mathcal{R} $ be the collection of sets which are modifications of range sets $ R_\lambda$ by sets of measure zero; that is, each element $X \in  \mathcal R$ has the form 
	\[ X = R_\lambda \cup S \quad \text{ or } \quad X = R_\lambda \backslash S\]
	for some set $S$ of measure zero. 
	Let $\sigma(\mathcal{R})$ be the $\sigma $-algebra generated by $\mathcal{R}.$
The representation $ \{t_\lambda\}_{\lambda \in \Lambda}$ 
 is monic, with cyclic vector $\chi_X \in L^2(X, \mathcal F, \mu)$, if and only if $\sigma(\mathcal R) = \mathcal F$. 
		\label{prop-conv-to-8-7}
In particular, for a monic representation $\{t_\lambda\}_{\lambda \in \Lambda}$,  the set
	\[
	{\mathcal{S} }:=  \Big\{ \sum_{i=1}^{n} a_{i}t_{\lambda_i} t_{\lambda_i}^* \chi_X \ | \ n \in \N , \lambda_i \in \Lambda, a_i \in \C\Big\} = \Big\{ \sum_{i=1}^{n} a_{i} \chi_{R_{\lambda_i}}  \ | \ \ n \in \N , \lambda_i \in \Lambda, a_i \in \C\Big\}
	\]
	is dense in  $L^2(X, \mathcal F, \mu)$.

\end{thm}

\begin{proof}

Suppose first that the representation $\{t_\lambda\}_{\lambda\in \Lambda}$ is monic and that $\chi_X$ is a cyclic vector for the representation.  As computed in the proof of Theorem 3.4 of \cite{FGKP}, we have 
\[ t_\lambda t_\lambda^* (\chi_X) = \chi_{R_\lambda}.\]
Therefore, our hypothesis that $\chi_X$ is a cyclic vector implies that for any $f \in L^2(X, \mathcal F, \mu)$, there is a sequence $(f_j)_j$, with $f_j \in \text{span} \{ \chi_{R_\lambda}: \lambda \in \Lambda\}$, such that 
\[\lim_{j\to \infty}  \int_X | f_j - f|^2 \, d\mu =0.\]
In particular, $(f_j) \to f$ in measure.  

For any $\sigma$-algebra $\mathcal T$, standard measure-theoretic results \cite[Proposition 6]{Moore} imply that since $\mu(X) < \infty$, convergence in measure among $\mathcal T$-a.e.~finite measurable functions on $(X, \mathcal T, \mu)$ is metrized by the distance 
\[ d_{\mathcal T} (f, g) :=  \int_{\Omega} \frac{ | f-g| }{  1+ | f-g| } d\mu.\] 
Moreover, $d_{\mathcal T}$ makes the space of $\mathcal S$-a.e.~finite measurable functions into a complete metric space (this can be seen, for example, by combining
 Proposition 1 and Corollary 7 of \cite{Moore}). 

The fact that $(f_j)_j \to f$ in measure in $(X, \mathcal F, \mu)$, and that  $f_j \in L^2(X, \sigma(\mathcal R), \mu)$ for all $j$, implies that $(f_j)_j$ is a Cauchy sequence with respect to both $d_{\mathcal F}$ and $d_{\sigma(\mathcal R)}$.  Consequently, the limit $f$ of $(f_j)_j$ must also be a $\sigma(\mathcal R)$-a.e.~finite measurable function.  In other words, every $f \in L^2(X, \mathcal F, \mu)$ is in fact in $L^2(X, \sigma(\mathcal R), \mu)$.  Since $\mathcal R \subseteq \mathcal F$ by construction we must have $\sigma(\mathcal R) = \mathcal F$, as desired.

For the converse, assume $\sigma(\mathcal R) = \mathcal F$.
We begin by observing  that
$$
\tilde{\mathcal{R}} := \{ \text{finite unions of elements in $\mathcal{R}$} \}
$$ is a subalgebra of $ \mathcal{P}(X)$ -- that is, closed under finite unions and complements. 
Closure under finite unions follows from the definition, while the second claim follows from  Equation \eqref{eq-partition}. 
Moreover, $\sigma(\mathcal{R})= \sigma(\mathcal{\tilde{R}}) = \mathcal F,$ so 
\[
{\tilde{\mathcal{S}}}:=  \Big\{ \sum_i^{n} a_{i} \chi_{{B_i}}  \ | \ \ n \in \N , B_i \in \sigma(\mathcal{\tilde{R}}), a_i \in \C\Big\}
\]
is dense in  $L^2(X, \mathcal F, \mu)$. 
Therefore, the Carath\'eodory/Kolmogorov extension theorem implies  that the measure $\mu|_{\mathcal{\tilde{R}}}$ restricted to $\tilde{\mathcal{R}}$ induces a unique (extended) measure on $\mathcal{F} = \sigma(\mathcal R)$, which we still call $\mu$. (This is indeed the original measure on $L^2(X,\mathcal{F}, \mu)$ by the uniqueness of the extension.)


To show that the vector 
	 $\chi_X$ is monic, equivalently that the set $\mathcal{S}$ is dense in $L^2(X, \mathcal{F}, \mu)$ equipped with the usual metric $ d_{L^2(X, \mathcal{F}, \mu) }$ coming from the $L^2$ norm, 
	we invoke a standard fact about metric spaces: if $(Q,d_Q)$ is a metric space, and if $\tilde{\Sigma} \subseteq Q$  is a dense subset of $(Q,d_Q)$, then any other subset ${\Sigma} \subseteq Q$ having the property
	\[
	\forall \epsilon >0, \ \forall \tilde{x}\in \tilde{\Sigma} ,\ \exists \  x_\epsilon \in {\Sigma} \text{ with } d_Q(\tilde{x}, x_\epsilon ) <\epsilon
	\]
	is also dense in $(Q,d_Q)$.
	We wish  to apply this fact in the setting where
	\[
	(Q,d_Q) = (L^2(X, \mathcal{F}, \mu), d_{L^2(X, \mathcal{F}, \mu) }),\text{ with } \tilde{\Sigma} = \tilde{\mathcal{S}},\ {\Sigma} = {\mathcal{S}}.
	\]
	
Choose $\tilde s \in\tilde{ \mathcal S}$ and fix $\epsilon > 0$.  Without loss of generality we can assume 
	$
\tilde{s}=  \sum_i^{n} a_{i} \chi_{{B_i}} $ for some $ n \in \N , B_i \in \sigma(\mathcal{\tilde{R}}), a_i \not=0 $.
Define
	$
	A:=\sum_i^{n} |a_{i} |\in \R.
	$
	The Carath\'eodory/Kolmogorov extension theorem\footnote{See \cite{BDurrett} Page 452, Appendix: Measure theory, Exercise 3.1.} also guarantees that  for any $\tilde{B} \in \sigma(\mathcal{ \tilde{\mathcal{R}}} ) = \mathcal F$ and for any $\epsilon >0,$ there exists $ A_\epsilon^{\tilde{B}}\in \mathcal{ \tilde{\mathcal{R}}} $ with 
	\[
	\mu\Big( \tilde{B} \Delta  A_\epsilon^{\tilde{B}} \Big) < \epsilon,
	\]	
	where $\Delta$ denotes symmetric difference.
In other words, for each $i$, there exists  $A_i^\epsilon \in \tilde{\mathcal{R}} $ such that
	\[
	\mu(B_i \Delta A_i^\epsilon   )<\frac{\epsilon^2}{A^2},\text{ or, equivalently,  }\Big(\mu(B_i \Delta A_i^\epsilon   )\Big)^{1/2}<\frac{\epsilon}{A}.
	\]
Thus, setting  $
	s_\epsilon := \sum_i^{n} a_{i} \chi_{A_i^\epsilon} $
	and using the triangle inequality yields that $
	 d_{L^2(X, \mathcal{F}, \mu) }(\tilde{s}, 	s_\epsilon    ) <\epsilon,$
	as desired.
\end{proof}

\begin{rmk}
Using the characterization of monic representations from Theorem \ref{prop-conv-to-8-7}, it is straightforward to check that the $\Lambda$-semibranching function systems detailed in  \cite[Example 3.5 and Section 4]{FGJorKP} generate monic representations of $C^*(\Lambda)$.  Similarly, suppose $\Lambda = \Lambda_1 \times \Lambda_2$ is a product $k$-graph and we have $\Lambda_i$-semibranching function systems on measure spaces $(X_i, \mu_i)$ for $i=1,2$, such that the associated $\Lambda_i$-semibranching representations are monic.   Then Theorem \ref{prop-conv-to-8-7} combines with {\cite[Proposition 3.4]{FGJorKP}} to tell us that the product $\Lambda$-semibranching function system on $(X_1 \times X_2, \mu_1 \times \mu_2)$ also gives rise to a monic representation of $C^*(\Lambda)$.
\end{rmk}
 We now proceed to  analyze several other examples of representations arising from $\Lambda$-semibranching function systems and establish which ones are monic representations. 
\begin{example} 
\label{ex-exonevthreeed-monic}
We present here an example of a $\Lambda$-semibranching  representation on a 1-graph  that   is not monic. 
The $1$-graph $\Lambda$ has  two vertices $v_1$ and $v_2$ and three edges $f_1,f_2$ and $f_3$.
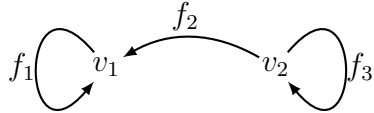
\[
\begin{tikzpicture}[scale=1.5]
 \node[inner sep=0.5pt, circle] (v) at (0,0) {$v_1$};
    \node[inner sep=0.5pt, circle] (w) at (1.5,0) {$v_2$};
    \draw[-latex, thick] (w) edge [out=50, in=-50, loop, min distance=30, looseness=2.5] (w);
    \draw[-latex, thick] (v) edge [out=130, in=230, loop, min distance=30, looseness=2.5] (v);
\draw[-latex, thick] (w) edge [out=150, in=30] (v);
\node at (-0.75, 0) {\color{black} $f_1$}; 
\node at (0.7, 0.45) {\color{black} $f_2$};
\node at (2.25, 0) {\color{black} $f_3$};
\end{tikzpicture}
\]
Let $X$ be the closed unit interval $[0,1]$ of $\R$ with the usual  Lebesgue $\sigma$-algebra and measure $\mu$. For $v_1$ and $v_2$, let $D_{v_1}=[0, \frac{1}{2}]$ and $D_{v_2}=(\frac{1}{2}, 1]$. Also for each edge $f\in \Lambda$, let $D_{f}=D_{s(f)}$, and hence $D_{f_1}=D_{v_1}=[0,\frac{1}{2}]$,  $D_{f_2}=D_{v_2}=(\frac{1}{2}, 1]$ and $D_{f_3}=D_{v_2}=(\frac{1}{2}, 1]$.
Now define prefixing maps for $f_1, f_2$ and $f_3$ by
\[\begin{split}
\tau_{f_1}(x)=-\frac{1}{2}x+\frac{1}{2}\quad\quad &\text{for $x\in D_{f_1}=\big[0,\frac{1}{2}\big]$},\\
\tau_{f_2}(x)=-\frac{1}{2}x+\frac{1}{2}\quad\quad&\text{for $x\in D_{f_2}=\big(\frac{1}{2},1\big]$},\\
\tau_{f_3}(x)=x\quad\quad &\text{for $x\in D_{f_3}=\big(\frac{1}{2},1\big]$}.
\end{split}\]
Then $R_{f_1}=\big[\frac{1}{4},\frac{1}{2}\big]$, $R_{f_2}=\big[0,\frac{1}{4}\big)$ and $R_{f_3}=\big(\frac{1}{2},1\big]$. Then the ranges of the prefixing maps are mutually disjoint and $X=R_{f_1}\cup R_{f_2}\cup R_{f_3}$.
For each $f_i$, since Lebesgue measure is regular, the Radon-Nikodym derivative of $\tau_{f_i}$ 
is given by
\begin{equation*}
\Phi_{{f_i}}(x) = \inf_{x \in E \subseteq D_{f_i}} \frac{(\mu\circ \tau_{f_i})(E)}{\mu(E)}\\
= \inf_{x \in E \subseteq D_{f_i}}\left\{ \begin{array}{cl}\frac{\frac{1}{2} \mu(E)}{\mu(E)}, & i=1,2, \\
\frac{\mu(E)}{\mu(E)}, & i=3
\end{array}\right.
\end{equation*}
\begin{equation*}
  = \left\{ \begin{array}{cl} \frac{1}{2}, & i =1,2 \\
1, & i=3.
\end{array}\right.
\end{equation*}
Now define $\tau^1:X\to X$ by
\[
\tau^1(x)=\begin{cases} \tau^{-1}_{f_1}(x)\quad\text{for $x\in R_{f_1}$}\\
\tau^{-1}_{f_2}(x)\quad\text{for $x\in R_{f_2}$} \\
\tau^{-1}_{f_3}(x)\quad\text{for $x\in R_{f_3}$} \end{cases}
\]
Since the sets $R_{f_i}$ are mutually disjoint, $\tau^1$ is well defined on $X$.
Then $\tau^1$ is the coding map satisfying $\tau^1\circ \tau_{f_i}(x)=x$ for all $x\in D_{f_i}$. 

It is a straightforward calculation to check that $\{\tau_{f_i}:D_{f_i}\to R_{f_i}, i=1,2,3\}$ is a semibranching function system for $(X,\mu)$. 
 To see that this this $\Lambda$-semibranching function system does not give rise to a monic representation, we argue by contradiction.  First, observe that the only finite paths with range $v_2$ are of the form $f_3 f_3 \cdots f_3$; and since $\tau_{f_3}(x)=x$ on $D_3=(1/2, 1]$, we have
\[R_{f_3} = R_{f_3 f_3 \cdots f_3} = (1/2, 1].\]
Every other finite path $\lambda$, having range $v_1$, will satisfy $R_\lambda \subseteq D_{v_1} = [0, 1/2]$.

Consequently, $\mathcal R = \{R_\lambda\}_{\lambda\in \Lambda}$ does not generate the usual Lebesgue $\sigma$-algebra on $[0,1]$, even after modification by sets of measure zero, since the restriction of $\mathcal R$  to $(1/2, 1]$ contains no nontrivial measurable sets.  Theorem \ref{prop:range-sets-generate-sigma-alg-in-monic} therefore
implies that the representation of $C^*(\Lambda)$ associated to this $\Lambda$-semibranching function system is not monic, and {hence is not equivalent to any representation on $L^2(\Lambda^\infty,\mu)$ arising from a $\Lambda$-projective system.}

\end{example}  

\begin{rmk}
We observe that since monic representations are multiplicity free,  it is easy to construct further examples of non-monic representations by using direct sums of monic representations, see  \cite{arveson} page 54.
\end{rmk}

In order to describe the following example of a $\Lambda$-semibranching representation which is monic, we review the concept of a Markov measure (see \cite[Section 3.1]{dutkay-jorgensen-monic} {or \cite[Section 4.2]{FGJorKP}} for more details, or  \cite{bezuglyi-jorgensen} for Markov measures in a more general context).

\begin{defn}[Definition 3.1 of  \cite{dutkay-jorgensen-monic}]
		\label{def-Markov-measure} 
A {\em Markov measure} on the infinite path space  $\Lambda^\infty_{\mathcal{O}_N}$
\[
\Lambda^\infty_{\mathcal{O}_N}=\prod_{i=1}^\infty \Z_N=\{(i_1 i_2\dots)\,:\, i_n\in \Z_N,\;\; n=1,2,\dots\}.
\] of the Cuntz algebra  ${\mathcal{O}_N}$ is defined by a vector 
	$\lambda = (\lambda_0, \ldots, \lambda_{N-1}) $ and an $N \times N$ matrix $T$ such that $\lambda_i > 0$, $T_{i,j} > 0$ for all $i,j \in \Z_N,$ and if 
	$e = (1,1, \ldots ,1)^t$ then $\lambda T=\lambda $ and $Te=e$.
The Carath\'eodory/Kolmogorov extension theorem then implies that  there exists a unique   Borel measure $\mu$ on $\Lambda^\infty_{\mathcal{O}_N}$  extending the measure $\mu_{\mathcal{C}}$ defined on cylinder sets by:
\begin{equation}
\mu_{\mathcal{C}} (Z(I)) : = \lambda_{i_1} T_{i_1,i_2} \cdots T_{i_{n-1},i_n},\text{ if }I = i_1 \ldots i_n.
\label{eq-def-markov-measu-Cuntz}
\end{equation}
The extension  $\mu$ is called a \emph{Markov measure} on $\Lambda^\infty_{\mathcal{O}_N}$.
  \end{defn}

For $N=2,$  
given a number $ x \in (0,1)$, we can take  $T=T_x= \begin{pmatrix}
 x & (1-x) \\ (1-x)  & x 
\end{pmatrix}$, and  $\lambda=(1,1)$. The resulting measure will in this case be called $\mu_x$. 
Moreover, if $x \not= x'$, Theorem 3.9 of \cite{dutkay-jorgensen-monic} guarantees that $\mu_x, \mu_{x'}$ are mutually singular.

\begin{example}
\label{ex:markov-is-monic}
We now consider an example of    $\Lambda$-semibranching function system which does give rise to a monic representation. 
Let $\Lambda$ be the 2-graph
below.
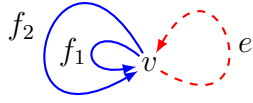
\[
\begin{tikzpicture}[scale=1.7]
 \node[inner sep=0.5pt, circle] (v) at (0,0) {$v$};
\draw[-latex, blue, thick] (v) edge [out=140, in=190, loop, min distance=15, looseness=2.5] (v);
\draw[-latex, blue, thick] (v) edge [out=120, in=210, loop, min distance=40, looseness=2.5] (v);
\draw[-latex, red, thick, dashed] (v) edge [out=-30, in=60, loop, min distance=30, looseness=2.5] (v);
\node at (-0.6, 0.1) {$f_1$}; \node at (-1,0.3) {$f_2$}; \node at (0.75,0.15) {$e$};
\end{tikzpicture}
\]
{
Recall from Remark \ref{rmk:rainbow} that every infinite path in a $2$-graph can be uniquely written as an  infinite string of composable edges which alternate in color: red, blue, red, \ldots.}
It follows that the infinite path space of the above 2-graph is homeomorphic to $\Lambda^\infty_{\mathcal{O}_2} \cong \prod_{i=1}^\infty \Z_2$ via the identification 
\[ e f_{j_1} e f_{j_2} e f_{j_3} \cdots \mapsto j_1 j_2 j_3 \cdots.\]
Therefore, the measure $\mu_x$ described above can be viewed as a measure  on $\Lambda^\infty$.



 It is straightforward to check that, as operators on $L^2(\Lambda^\infty, \mu_{x})$, the prefixing operators  $\sigma_e,\sigma_{f_1},\sigma_{f_2}$ have positive Radon-Nikodym derivatives at any point $z \in \Lambda^\infty$.
Consequently, the standard prefixing and coding maps make $(\Lambda^\infty, \mu_x)$ into a $\Lambda$-semibranching function system.  The associated representation of $C^*(\Lambda)$ is therefore monic, by Theorems \ref{thm-characterization-monic-repres} and Theorem~\ref{thm:SBFS-repn}, and Example \ref{ex:standard-SBFS-is-proj}.


\end{example}

\section{A universal representation for non-negative $\Lambda$-projective systems} 
\label{sec:univ_repn}

The focus of this section is the construction of a \lq universal representation' of $C^*(\Lambda)$, generalizing the work of \cite[Section 4]{dutkay-jorgensen-monic} for the  Cuntz algebra setting, such that  every non-negative monic representation of $C^*(\Lambda)$ is a sub-representation of the universal representation.  
The Hilbert space $\mathcal{H}(\Lambda^\infty)$ on which our universal representation is defined is
the \lq universal space' for representations of $C(\Lambda^\infty)$, see \cite{nelson}, and also  \cite{dutkay-jorgensen-monic,bezuglyi-jorgensen-infinite,alpay-jorgensen-Lew,palle-iterated}.
For the case of $\mathcal{O}_N,$ this space was also shown to be the  \lq universal representation space'
for monic representations in \cite{dutkay-jorgensen-monic}.  We recall the construction of $\mathcal{H}(\Lambda^\infty)$ below.

\begin{defn} 
\label{def-universal-Hilbert-space}
Let $\Lambda$ be a finite 
 $k$-graph with no sources, and let $\Lambda^\infty$ be the infinite path space of $\Lambda,$ endowed with the topology generated by the cylinder sets and the Borel $\sigma$-algebra associated to it.
 
  Consider the collection of pairs $(f, \mu)$, where $\mu$ is a Borel measure on $\Lambda^\infty$, and  $f\in L^2(\Lambda^\infty,\mu)$. 
 We say that two pairs $(f,\mu)$ and $(g,\nu)$ are equivalent, denoted by $(f,\mu)\sim (g,\nu)$, if there exists a finite Borel measure $m$ on $\Lambda^\infty$ such that
\[
\mu << m,\ \nu << m,\ \hbox{ and }\;\; f \sqrt{\frac{d\mu}{dm}} = g \sqrt{\frac{d\nu}{dm}} \;\; \text{ in } \, L^2(\Lambda^\infty, m).
\]

We write $f \, \sqrt{d\mu}$ for the equivalence class of $(f,\mu)$.

\end{defn}

Proposition~8.3 of \cite{bezuglyi-jorgensen-infinite} establishes that $\mathcal{H}(\Lambda^\infty)$ is a Hilbert space, with the vector space structure given by scalar multiplication and 
\[  f\,\sqrt{d\mu} +  g\,\sqrt{d\nu} := \Big( f \sqrt{\frac{d\mu}{d(\mu + \nu)}}+g \sqrt{\frac{d\nu}{d( \mu + \nu)}}\Big) \sqrt{{d(\mu + \nu)}},
\]
and the inner product given by 
\begin{equation}
\label{eq:Hilbert-sp}
\langle f\,\sqrt{d\mu} ,  g\,\sqrt{d\nu}\rangle := \int_{\Lambda^\infty} \overline{f} g\, 
\Big( \sqrt{\frac{d\mu}{d(\mu + \nu)}}\ \sqrt{\frac{d\nu}{d(\mu + \nu) }}\Big) {{d(\mu + \nu)}}.
\end{equation}
We  call $\mathcal{H}(\Lambda^\infty)$ the \emph{universal Hilbert space}  for $\Lambda^\infty$.


The following fundamental property of $\mathcal{H}(\Lambda^\infty)$ justifies the name `universal Hilbert space.' 

\begin{prop}(\cite[Theorem~3.1]{palle-iterated}, \cite{dutkay-jorgensen-monic}, \cite{alpay-jorgensen-Lew})
\label{prop-dutkay-jorgensen-lemma-4.6}
Let $\Lambda$ be a finite $k$-graph with no sources.
For every finite Borel measure $\mu$ on $\Lambda^\infty$, define $W_\mu:L^2(\Lambda^\infty, \mu) \to \mathcal{H}(\Lambda^\infty)$ by
$
W_\mu(f) = f \sqrt{d\mu}. 
$
Then $W_\mu$ is an isometry of $L^2(\Lambda^\infty, \mu)$ onto a subspace of $\mathcal{H}(\Lambda^\infty)$, which we call $\mathcal{L}^2(\mu)$. 
\end{prop}

We are now ready to present the universal representation $\pi_{univ}$  of $C^*(\Lambda)$ on $\mathcal{H}(\Lambda^\infty)$.

\begin{prop}
\label{prop-univ-repres}
\label{univ-repres-definition-result}
Let $\Lambda$ be a finite  $k$-graph with no sources. Fix $(f,\mu)\in \mathcal{H}(\Lambda^\infty)$.  For each $\lambda \in \Lambda^n$, define  $S^{univ} \in B(\mathcal H(\Lambda^\infty))$ by 
\[
 S_\lambda^{univ} (f\,\sqrt{d\mu}) : = (f \circ \sigma^n) \,\sqrt{d(\mu \circ \sigma_\lambda^{-1})},
\]
where $\sigma_\lambda$ and $\sigma^n$ are the standard prefixing  and coding maps  
of Definition \ref{def:infinite-path}.
Then:
\begin{itemize}
\item[(a)] The adjoint of $S_\lambda^{univ}$ is given by
$
 (S_\lambda^{univ})^* (f\,\sqrt{d\mu}) : = (f \circ \sigma_\lambda) \,\sqrt{d(\mu \circ \sigma_\lambda)}.
$
\item[(b)] The operators $\{S_\lambda^{univ}: \lambda \in \Lambda\}$ generate a representation $\pi_{univ}$ of $C^*(\Lambda)$ on $\H(\Lambda^\infty)$, which we call the \lq universal  representation'. 
\item[(c)] The projection valued measure $P$ on $\Lambda^\infty$ given in Definition \ref{def-proj-val-measu} associated to  the universal representation $\pi_{univ}$ is
given by:
\begin{equation}
\label{proj-valued-measure-universal-representation}
P(A)(f\,\sqrt{d\mu}) = (\chi_A \cdot f)\,\sqrt{d\mu} ,
\end{equation}
where $A$ is a Borel set of the Borel $\sigma$-algebra generated by the cylinder sets.
\end{itemize}
\end{prop}

\begin{proof}
The proof is similar to that of  Proposition 4.2 of   \cite{dutkay-jorgensen-monic}, although the details are more involved because of the more complicated $k$-graph structure. 
To simplify the notation, in this proof we will  drop the superscript $univ$ from $S^{univ}_\lambda$. 
To  check that $\{S_\lambda^{univ}\}_{\lambda \in \Lambda}$ gives a representation of $C^*(\Lambda)$,
first we observe that the operators $S_\lambda$ are well defined; in other words, if $f \sqrt{d\mu} = g \sqrt{d\nu}$, we must have $(f \circ \sigma^n) \sqrt{d(\mu \circ \sigma_\lambda^{-1})} = (g \circ \sigma^n)\sqrt{d(\nu \circ \sigma_\lambda^{-1})}$ for all $\lambda \in \Lambda$.

Suppose that  $f \sqrt{d\mu} = g \sqrt{d\nu}$. 
{
Observe that $\mu \circ \sigma_\lambda^{-1}$ is zero off $Z(\lambda)$, and $\mu \circ \sigma_\lambda^{-1} =  \left(m|_{Z(s(\lambda))} \circ \sigma_\lambda^{-1}\right)$ on $Z(\lambda)$. 
The fact that $\sigma_\lambda^{-1} = \sigma^{d(\lambda)}$ on $Z(\lambda)$ now implies that, if $n := d(\lambda)$,
\begin{align*}
&(f \circ \sigma^n) \sqrt{d(\mu \circ \sigma_\lambda^{-1})} = \left( (f \circ \sigma^n) \sqrt{d(\mu \circ \sigma_\lambda^{-1})}\right)|_{Z(\lambda)} = \left(f \sqrt{d\mu}\right) |_{Z(s(\lambda))} \circ \sigma_\lambda^{-1} \\
&= \left(g \sqrt{d\nu}\right) |_{Z(s(\lambda))} \circ \sigma_\lambda^{-1}= \left( (g\circ \sigma^n) \sqrt{d(\nu \circ \sigma_\lambda^{-1})}\right)|_{Z(\lambda)}= (g\circ \sigma^n) \sqrt{d(\nu \circ \sigma_\lambda^{-1})}.
\end{align*}
It follows that $S_\lambda$ is well defined.}

To check the formula for $S_\lambda^*$ given in the statement of the proposition, we compute:
\begin{align*}
&\langle S_\lambda^* (f \sqrt{d\mu}),  g \sqrt{d\nu}\rangle = \langle f \sqrt{d\mu}, S_\lambda g \sqrt{d\nu} \rangle = \langle f \sqrt{d\mu}, (g \circ \sigma^n) \sqrt{d(\nu \circ \sigma_\lambda^{-1})} \rangle \\
&= \int_{\Lambda^\infty} \overline{f(x)} ( g \circ \sigma^n)(x) \sqrt{\frac{d\mu}{d(\mu + (\nu \circ \sigma_\lambda^{-1}))}} \sqrt{\frac{d(\nu \circ \sigma_\lambda^{-1})}{d(\mu + (\nu \circ \sigma_\lambda^{-1}))}}\, d(\mu + (\nu \circ \sigma_\lambda^{-1})).
\end{align*}
This integral vanishes off $Z(\lambda)$, since $\sigma_\lambda^{-1}$ (and consequently $d(\nu \circ \sigma_\lambda^{-1})$) do.  We thus use the fact that $(\nu \circ \sigma_\lambda^{-1})|_{Z(\lambda)} = \nu|_{Z(s(\lambda))} \circ \sigma_\lambda^{-1}$ to rewrite 
\begin{align*}
\langle S_\lambda^* (f \sqrt{d\mu}), g \sqrt{d\nu}\rangle &= \int_{Z(s(\lambda))} \overline{f \circ \sigma_\lambda(x)} g(x) \sqrt{\frac{d(\mu \circ \sigma_\lambda)}{d((\mu \circ \sigma_\lambda) + \nu)}} \sqrt{\frac{d\nu}{d((\mu \circ \sigma_\lambda)+ \nu)}} d((\mu \circ \sigma_\lambda) + \nu).
\end{align*}
Hence $S_\lambda^*(f\sqrt{d\mu}) = (f \circ \sigma_\lambda) \sqrt{d(\mu \circ \sigma_\lambda)}$, which proves (a).

{
Checking condition (b), that the operators $\{S_\lambda\}$ give a representation of $C^*(\Lambda)$, is a straightforward computation, analogous to the proof of Proposition \ref{prop:lambda-proj-repn}.}

To see (c), note that Equation \eqref{proj-valued-measure-universal-representation} follows from the observation that $S_\nu S_\nu^*$ acts by multiplication by $\chi_{Z(\nu)}$; the fact that disjoint unions of cylinder sets $Z(\nu)$ generate the $\sigma$-algebra up to sets of measure zero \cite[Lemma 4.1]{FGKP} therefore enables us  to compute $P(A)$ by linearity, for any Borel set $A$.
\end{proof}

The following two Propositions, which detail additional technical properties of the projection valued measure associated to $\pi_{univ}$,  will be used  in the proof of Theorem \ref{prop-universal-name}, the main result of this Section.

\begin{prop} 
\label{prop-commutant-dutkay-jorgensen-lemma-4.4}
Let $\Lambda$ be a finite $k$-graph with no sources and let $\mathcal{H}(\Lambda^\infty)$ be the Hilbert space described in Definition~\ref{def-universal-Hilbert-space}, and let $\pi_{univ}=  \{S_\lambda^{univ}: \lambda \in \Lambda\}$ be the universal representation of $C^*(\Lambda)$ on $\mathcal{H}(\Lambda^\infty)$ given in Proposition~\ref{prop-univ-repres}.

\begin{itemize}
\item[(a)] For $y\in \mathcal{H}(\Lambda^\infty)$, define a function $\nu_y$ on $\Lambda^\infty$ by
\[
\nu_{y} (Z(\lambda) ) : =  \langle (S^{univ}_\lambda (S^{univ}_\lambda)^*) y, y \rangle,
\]
where $\langle \cdot,\cdot\rangle$ is the inner product given on $\mathcal{H}(\Lambda^\infty)$ in Equation~\eqref{eq:Hilbert-sp}. Then $\nu_y$ gives a measure on $\Lambda^\infty$.

\item[(b)]  Let $T$ be a bounded operator on $\mathcal{H}(\Lambda^\infty)$.  If  $T$ commutes with $\pi_{univ}|_{C(\Lambda^\infty)}$,  then for any  $x\in \mathcal{H}(\Lambda^\infty)$ we have
\[
\nu_{T(x)} << \nu_{x}.
\]
\item[(c)] For every   vector
$f \sqrt{d \mu}\in \mathcal{H}(\Lambda^\infty)$, we have
$
 \nu_{f \sqrt{d \mu}} = |f|^2  \mu.
$
\end{itemize}
\end{prop}

\begin{proof} As in Equation \eqref{eq:monic_measure}, it is straightforward to see (a). For (b), fix $x\in \mathcal{H}(\Lambda^\infty)$. Then since $T$ commutes with $S^{univ}_\lambda$, we have
 \[
\nu_{T(x)}(Z(\lambda) ) =  \langle (S^{univ}_\lambda (S^{univ}_\lambda)^*)T(x), T(x) \rangle =  \langle (S^{univ}_\lambda (S^{univ}_\lambda)^*)x, T^* T (x) \rangle.
 \]
Since each $S^{univ}_\lambda$ is a partial isometry, the Cauchy-Schwarz inequality then gives
  \[
 \nu_{T(x)}(Z(\lambda) )^2 \leq  \Vert(S^{univ}_\lambda (S^{univ}_\lambda)^*)x\Vert^2 \ \Vert T^* T x \Vert^2 = \nu_{x}(Z(\lambda) )^2  \ \Vert T^* T x \Vert^2 ,
  \]
which gives that $\nu_T(x) << \nu_x$.

For (c), 
Equation \eqref{proj-valued-measure-universal-representation} implies that, for any cylinder set $Z(\eta)$, 
\[
\nu_{f \sqrt{d \mu}} (Z(\eta) )  =  \langle( {\chi_{Z(\eta)}} f)\,\sqrt{d\mu} , f \sqrt{d \mu} \rangle =\int \chi_{Z(\eta)} \cdot  |f|^2 d\mu \;  = \int_{Z(\eta)} |f|^2 \, d\mu.
\]
This
 gives the desired result. 
\end{proof}

We now present  an important result which will allow us to derive, in Theorem  \ref{prop-universal-name}, the desired universal property of the representation.

\begin{thm} 
\label{them-universal-representation-commutant} 
Let $\Lambda$ be a finite $k$-graph with no sources. Let $\mathcal{H}(\Lambda^\infty)$ be the universal Hilbert space  for $\Lambda^\infty$ and $\pi_{univ}$ be the universal representation of $C^*(\Lambda)$ on $\mathcal{H}(\Lambda^\infty)$. Then:
\begin{enumerate}
\item [(a)] An operator $T\in \mathcal{B}(\mathcal{H}(\Lambda^\infty))$ commutes with  $\pi_{univ}|_{C(\Lambda^\infty)}$ 
if and only if for each finite Borel  measure $ \mu $ on $\Lambda^\infty$ which arises from a monic representation of $C^*(\Lambda)$ as in Equation \eqref{eq:monic_measure}, 
there exists a function $F_\mu$ in ${L}^\infty(\Lambda^\infty, \mu)$ such that: 
\begin{enumerate}
\item[(i)]  $sup \{ \Vert F_\mu\Vert : \mu \text{ arises from a monic representation }\} < \infty$.
\item[(ii)]  If $\mu << \lambda$ then $F_\mu = F_\lambda$, $\mu$-a.e.
\item[(iii)]  $T(f\sqrt{d\mu}) = F_\mu f\sqrt{d\mu}$ for all $f\sqrt{d\mu} \in \mathcal{H}(\Lambda^\infty)$ 
\end{enumerate}
\item[(b)]  Let $\mathcal H$ denote the subspace of $\mathcal H(\Lambda^\infty)$ spanned by vectors of the form $f \sqrt{d\mu}$ 
where $\mu$ arises from a monic representation.  An operator $T\in \mathcal{B}(\mathcal{H}(\Lambda^\infty))$  commutes with $\pi_{univ}|_\mathcal H$ 
if and only if  for every finite Borel measure $\mu$ on $\Lambda^\infty$ arising from a monic representation of $C^*(\Lambda)$,  
and for each $\lambda \in \Lambda$, we have
\[
F_\mu =F_{\mu\circ \sigma_\lambda^{-1}} \circ \sigma_\lambda,\  \mu-a.e.
\]
\end{enumerate}
\end{thm}

\begin{proof} 
{Recall from Proposition}~\ref{prop-dutkay-jorgensen-lemma-4.6} {the isometry $W_\mu$ of $L^2(\Lambda^\infty, \mu)$ onto $\mathcal{L}^2(\mu)$.  Throughout the proof, we will assume that the finite Borel measure $\mu$ arises from a monic representation.}
We first claim that if   $T$ commutes with  $\pi_{univ}|_{C(\Lambda^\infty)}$, then $T$ maps $\mathcal{L}^2(\mu)$
into itself. 
To prove this, let  $x = f\sqrt{d\mu}$ be in $\mathcal{L}^2(\mu)$,  
and let
$T(x) = g\sqrt{d \zeta}$ {for $(g,\zeta)\in \mathcal{H}(\Lambda^\infty)$.}
 Then Proposition \ref{prop-commutant-dutkay-jorgensen-lemma-4.4} (b) implies that
$
\nu_{T(x)} << \nu_{x}.
$
By Proposition  \ref{prop-commutant-dutkay-jorgensen-lemma-4.4} (c), we have
\[ 
\nu_{x} = |f|^2\mu, \;\;\text{ and }\;\; \nu_{T(x)} = |g|^2 \zeta.
\]
Therefore $|g|^2 \zeta << \mu $,  so by the Radon--Nikodym theorem there exists $h \geq 0$ in  $L^1(\Lambda^\infty, \mu)$
such that $|g|^2\, d\zeta = h\, d\mu$. Then
\[
|g| \sqrt{d\zeta} =\sqrt{ h} \sqrt{d\mu},\;\; \text{ and } \;\;  |g|\, g \sqrt{d\zeta} =g \, \sqrt{ h}\, \sqrt{d\mu}.
\]
If $g =0$ on some Borel set $A$, then $\sqrt{h} \sqrt{d\mu}(A) = 0$ also. Therefore, 
\[
g\sqrt{d\zeta} =\begin{cases}
\frac{ g \sqrt{h} }{|g|} \sqrt{d\mu} \; \in \mathcal{L}^2(\mu), & g \not= 0 \\ 
 0, & g=0
 \end{cases}
\]
which shows that $T$ maps $\mathcal{L}^2(\mu)$
into itself. 

{
We now make some computations regarding the relationship between an arbitrary monic representation $\pi$ and the universal representation $\pi_{univ}$.  
 {Note that Equation \eqref{proj-valued-measure-universal-representation} implies that} $\pi_{univ}(\psi) (f \sqrt{d\mu_\pi}) = \left( \psi \cdot f\right) \sqrt{d\mu_\pi}$ for any $f \in L^2(\Lambda^\infty, \mu_\pi)$.

On the other hand, since $\pi$ is a monic representation, $\pi(\chi_{Z(\lambda)}) f = \chi_{Z(\lambda)} \cdot f \in L^2(\Lambda^\infty, \mu_\pi)$ by Equation \eqref{eq:range-sets}.  Therefore, 
\[ \pi_{univ}(\psi)(f \sqrt{d\mu_\pi}) = \left[ \pi(\psi)(f) \right] \sqrt{d\mu_\pi}.\]

By hypothesis, $T$ commutes with $\pi_{univ}|_{C(\Lambda^\infty)}$. 
Since $T$ preserves $\mathcal L^2(\mu)$ for each measure $\mu$ arising from a monic representation, there must exist  $g \in \mathcal L^2(\mu_\pi)$ such that  $T(f \sqrt{d\mu_\pi}) = g \sqrt{d\mu_\pi}$.  Consequently, 
\begin{align*}
T[ \pi(\psi) f] \sqrt{d\mu_\pi} & = T \pi_{univ}(\psi) ( f \sqrt{d\mu_\pi}) = \pi_{univ}(\psi) T(f \sqrt{d\mu_\pi})\\
 &= \pi_{univ}(\psi)( g \sqrt{d\mu_\pi})
 = [\pi(\psi)(g)] \sqrt{d\mu_\pi}  = \pi(\psi) T \left( f \sqrt{d\mu_\pi}\right),
\end{align*}
so (identifying $\mathcal L^2(\mu_\pi)\subseteq \H(\Lambda^\infty)$ with $L^2(\Lambda^\infty, \mu_\pi)$) we see that $T$ commutes with $\pi(\psi)$ for all $\psi \in C(\Lambda^\infty)$.} 
%
%

 Therefore, we can pull-back  $T$ to an operator $\widetilde{T}$ on $L^2(\Lambda^\infty, \mu)$  that commutes with all of the multiplication operators $\{ M_f: f \in C(\Lambda^\infty)\}$.  The fact (cf.~\cite{vJones}) that the  maximal abelian subalgebra of $\mathcal{B}(L^2(\Lambda^\infty, \mu)) $, for any finite Borel measure $\mu$, is the sub-algebra $L^\infty(\Lambda^\infty, \mu)$ consisting of multiplication operators now implies that 
$\widetilde{T}$ must be a multiplication operator too.
In other words, there exists a function $F_\mu$ in $L^\infty(\Lambda^\infty, \mu)$ such that 
\begin{equation}\label{eq:T}
T(f\sqrt{d\mu}) =F_\mu\, f\, \sqrt{d\mu}
\end{equation}
for all $f\in L^2(\Lambda^\infty,\mu)$, establishing $\it{(iii)}$.
It remains to check the properties of the functions $F_\mu$. One immediately observes that 
\[
\Vert F_\mu \Vert_{L^\infty(\mu)} \leq  \Vert  T \Vert 
\]
and this implies $\it{(i)}$. To check $\it{(ii)}$, suppose $\mu << \lambda$.  Then, for all $f \in L^2(\Lambda^\infty, \mu)$, 
we have $f\sqrt{d\mu} = f \sqrt{d \mu/ d\lambda} \sqrt{d\lambda}$, and hence 
\[
T(f\sqrt{d\mu}) = T\left(f \sqrt{\frac{d \mu}{ d\lambda}} \sqrt{d\lambda}\right) \quad 
\Longrightarrow \quad  F_\mu f\sqrt{d\mu} = F_\lambda f \sqrt{\frac{d \mu}{d\lambda}} \sqrt{d\lambda} \]
Thus, as elements of $\mathcal L^2(\lambda)$,
$ F_\mu f \sqrt{d \mu/ d\lambda} = F_\lambda f \sqrt{d \mu/d\lambda}
$
for any $f \in L^2(\Lambda^\infty, \mu)$, which implies $F_\mu {d \mu/ d\lambda} = F_\lambda  {d \mu/ d\lambda} \;\; (\lambda-a.e.).$  It follows that, for any Borel set $A$,
\[ \int_A(  F_\mu - F_\lambda ) \, d\mu = \int_A (F_\mu - F_\lambda) \frac{d\mu}{d\lambda} d\lambda = 0,\]
so $F_\mu = F_\lambda$, $\mu$-a.e. 
This proves $\it{(ii)}$.

For the converse, assume that  $T$ is given on $\mathcal L^2(\mu)$ by a function $F_\mu\in L^\infty(\Lambda^\infty,\mu)$ satisfying  $\it{(i),(ii),(iii)}$, { i.e.~$T(f\sqrt{d\mu})=F_\mu f \sqrt{d\mu}$ for all $f \sqrt{d\mu} \in\mathcal{H}(\Lambda^\infty)$ such that $\mu$ arises from a monic representation.}
%
Then {\it{(i)}} implies that T is bounded with 
 $\Vert T \Vert \leq \sup_\mu\{  \Vert F_\mu \Vert_{L^\infty(\mu)} \}.$
Since $T$ acts as a multiplication operator on each $\mathcal L^2(\mu)$, Part (c) of Proposition \ref{univ-repres-definition-result} implies that  $T$ commutes with $P(A)$ for all Borel subsets $A$ and therefore $T$ commutes with the {restricted} universal representation, {$\pi_{univ}|_{C(\Lambda^\infty)}$, which proves (a).}
 
 To prove (b), note that if an operator $T \in \mathcal B(\H(\Lambda^\infty))$  commutes with the universal representation $\pi_{univ}$ of $C^*(\Lambda)$ on $\mathcal H$, then in particular $T$ commutes with $\pi_{univ}|_{C(\Lambda^\infty)}$ on $\mathcal H$, and hence $T(f \sqrt{d\mu}) = F_\mu f \sqrt{d\mu}$ is a multiplication operator on each $\mathcal L^2(\mu)$ when the  measure $\mu$ arises from a monic representation.  In particular, $T$ is normal (when restricted to $\mathcal H$).  Therefore, by the Fuglede--Putnam theorem,  $T|_{\mathcal H}$  commutes with $\pi_{univ}$ iff $T S_\lambda^{univ}|_{\mathcal H} = S_\lambda^{univ} T|_{\mathcal H}$ for all $\lambda\in \Lambda$. 
 Using the formulas for $S_\lambda^{univ}$ from Theorem \ref{them-universal-representation-commutant}, we see that $T|_{\mathcal H}$ commutes with $\pi_{univ}|_{\mathcal H}$ if and only if, 
for each $f \sqrt{d\mu}\in \mathcal{H}$ and $\lambda\in\Lambda^n$,
\[ 
 F_{\mu \circ \sigma_\lambda^{-1}}(f \circ \sigma^n)\sqrt{d\mu \circ \sigma_\lambda^{-1}} =  (F_\mu \circ \sigma^n)(f \circ \sigma^n)\sqrt{d\mu \circ \sigma_\lambda^{-1}} ,
\]
or equivalently, $
 F_{\mu \circ \sigma_\lambda^{-1}}=  (F_\mu \circ \sigma^n) \ \text{ for } \lambda\in\Lambda^n, \; (\mu \circ \sigma_\lambda^{-1}) - a.e .
$
for all measures $\mu$ arising from monic representations of $C^*(\Lambda)$.  
  Composing with $\sigma_\lambda$ gives the desired result of (b).
\end{proof}
 
%
%

{
\begin{defn}
\label{def:nonnegative}
A monic representation $\{t_\lambda\}_{\lambda \in \Lambda}$ of a finite, source-free $k$-graph $\Lambda$ is said to be {\em nonnegative} if the functions $\{f_\lambda\}_{\lambda \in \Lambda}$ of the associated $\Lambda$-projective system on $\Lambda^\infty$ are nonnegative a.e.
\end{defn}}

 The following result, a consequence of Theorem \ref{them-universal-representation-commutant},  proves that every nonnegative monic representation is equivalent to a sub-representation of $\{S^{univ}_\lambda\}_{\lambda \in \Lambda}$, justifying the name `universal representation' for $\{S^{univ}_\lambda\}_{\lambda \in \Lambda}$.

 \begin{thm}
 \label{prop-universal-name} 
 Let $\Lambda$ be a finite $k$-graph with no sources.
 Let $ \{t_\lambda\}_{\lambda \in \Lambda}$ be a nonnegative  monic representation of $C^*(\Lambda)$ on $L^2(\Lambda^\infty, \mu_\pi)$. 
 Let $W$ be the isometry from $L^2(\Lambda^\infty, \mu_\pi)$ onto $\mathcal{L}^2(\mu_\pi )$ given in  Proposition \ref{prop-dutkay-jorgensen-lemma-4.6}, so that $
  W f = f\sqrt{d\mu_\pi} .$
  Then $W$ intertwines $\{t_\lambda\}_{\lambda\in\Lambda}$  with  the sub-representation $\{S_\lambda^{univ}|_{\mathcal L^2(\mu_\pi)}\}_{\lambda\in\Lambda}$ of the universal representation $\{S_\lambda^{univ}\}_\lambda$.
 \end{thm} 
 
 \begin{proof} 
 By Theorem \ref{thm-characterization-monic-repres} {and Proposition~\ref{prop:lambda-proj-repn}}, we can assume that $t_\lambda$ is of the form 
 \[ t_\lambda(f) = f_\lambda \cdot (f \circ \sigma^{d(\lambda)}),\]
 {where, since $\{t_\lambda\}_{\lambda \in \Lambda}$ is assumed nonnegative, we may assume $f_\lambda=\sqrt{\frac{d(\mu_\pi\circ (\sigma_\lambda)^{-1})}{d\mu_\pi}}$.}
 By Theorem \ref{them-universal-representation-commutant}, 
 \[
 \begin{split}
 W (t_\lambda f) &= W( f_\lambda ( f\circ \sigma^{d(\lambda)} ))=  f_\lambda ( f\circ \sigma^{d(\lambda)} )\sqrt{d\mu_\pi}= ( f\circ \sigma^{d(\lambda)} )\sqrt{| f_\lambda|^2d\mu_\pi}\\
 &= ( f\circ \sigma^{d(\lambda)} )\sqrt{d [\mu_\pi \circ (\sigma_\lambda)^{-1}] }= S_\lambda^{univ}(f\sqrt{d\mu_\pi})=S_\lambda^{univ}W(f).
 \end{split}
 \]
 In other words, $W$ intertwines $\{t_\lambda\}_{\lambda\in \Lambda}$ and $\{S_\lambda^{univ}|_{\mathcal L^2(\mu_\pi)}\}_{\lambda\in\Lambda}$, as claimed.
 \end{proof}

\end{document}